\documentclass[12pt]{article}
\usepackage[english]{babel}
\usepackage[top = 1in, bottom = 1in, left = 1in, right = 1in]{geometry}
\usepackage{amsmath,amsxtra,amssymb,latexsym,amscd,amsthm,enumitem,calligra,mathrsfs}
\usepackage{graphicx}
\usepackage[mathscr]{eucal}
\usepackage{fancybox}
\usepackage{xparse}
\usepackage{environ}
\usepackage[all]{xy}
\usepackage{tikz}
\usepackage{tikz-cd}
\usepackage{amscd}
\usepackage{mathtools}
\usepackage[notref,notcite]{showkeys}
\usepackage{stmaryrd}
\usetikzlibrary{positioning}
\usetikzlibrary{arrows}
\usepackage{hyperref}
\usepackage[nottoc]{tocbibind}
\usepackage[bibstyle=alphabetic, natbib, backend=bibtex]{biblatex}

\usepackage{imakeidx}
\makeindex[columns=2, title=Index, intoc]

\ExplSyntaxOn
\NewEnviron{pmatrixT}
 {
  \marine_transpose:V \BODY
 }

\int_new:N \l_marine_transpose_row_int
\int_new:N \l_marine_transpose_col_int
\seq_new:N \l_marine_transpose_rows_seq
\seq_new:N \l_marine_transpose_arow_seq
\prop_new:N \l_marine_transpose_matrix_prop
\tl_new:N \l_marine_transpose_last_tl
\tl_new:N \l_marine_transpose_body_tl

\cs_new_protected:Nn \marine_transpose:n
 {
  \seq_set_split:Nnn \l_marine_transpose_rows_seq { \\ } { #1 }
  \int_zero:N \l_marine_transpose_row_int
  \prop_clear:N \l_marine_transpose_matrix_prop
  \seq_map_inline:Nn \l_marine_transpose_rows_seq
   {
    \int_incr:N \l_marine_transpose_row_int
    \int_zero:N \l_marine_transpose_col_int
    \seq_set_split:Nnn \l_marine_transpose_arow_seq { & } { ##1 }
    \seq_map_inline:Nn \l_marine_transpose_arow_seq
     {
      \int_incr:N \l_marine_transpose_col_int
      \prop_put:Nxn \l_marine_transpose_matrix_prop
       {
        \int_to_arabic:n { \l_marine_transpose_row_int }
        ,
        \int_to_arabic:n { \l_marine_transpose_col_int }
       }
       { ####1 }
     }
   }
   \tl_clear:N \l_marine_transpose_body_tl
   \int_step_inline:nnnn { 1 } { 1 } { \l_marine_transpose_col_int }
    {
     \int_step_inline:nnnn { 1 } { 1 } { \l_marine_transpose_row_int }
      {
       \tl_put_right:Nx \l_marine_transpose_body_tl
        {
         \prop_item:Nn \l_marine_transpose_matrix_prop { ####1,##1 }
         \int_compare:nF { ####1 = \l_marine_transpose_row_int } { & }
        }
      }
     \tl_put_right:Nn \l_marine_transpose_body_tl { \\ }
    }
   \begin{pmatrix}
   \l_marine_transpose_body_tl
   \end{pmatrix}
 }
\cs_generate_variant:Nn \marine_transpose:n { V }
\cs_generate_variant:Nn \prop_put:Nnn { Nx }
\ExplSyntaxOff

\DeclareFontFamily{U}{wncy}{}
\DeclareFontShape{U}{wncy}{m}{n}{<->wncyr10}{}
\DeclareSymbolFont{mcy}{U}{wncy}{m}{n}
\DeclareMathSymbol{\Sh}{\mathord}{mcy}{"58}

\newcommand{\Rbb}{\mathbb{R}}

\newcommand{\Fbb}{\mathbb{F}}

\newcommand{\Qbb}{\mathbb{Q}}
\newcommand{\Zbb}{\mathbb{Z}}
\newcommand{\Fpbb}{\mathbb{F}_{p}}

\newcommand{\Qpbb}{\mathbb{Q}_{p}}
\newcommand{\Zpbb}{\mathbb{Z}_{p}}
\newcommand{\Pbb}{\mathbb{P}}
\newcommand{\Abb}{\mathbb{A}}
\newcommand{\Gmbb}{\mathbb{G}_{m}}

\newcommand{\ds}{\displaystyle}

\theoremstyle{definition}
\newtheorem{definition}{Definition}[section]
\newtheorem{remark}[definition]{Remark}
\newtheorem{example}[definition]{Example}

\theoremstyle{plain}
\newtheorem{theorem}{Theorem}[section]
\newtheorem{proposition}[theorem]{Proposition}

\newtheorem{corollary}[theorem]{Corollary}
\newtheorem{conjecture}[theorem]{Conjecture}

\begin{document}



\title{\textbf{Brauer--Manin obstruction for Wehler K3 surfaces of Markoff type}}
	\author{\textsc{Quang-Duc DAO}} 
	\date{}
	\maketitle

\begin{abstract}
Following recent work by E. Fuchs \textit{et al.} \cite{FLST22}, we study the Brauer--Manin obstruction for integral points on Wehler K3 surfaces of Markoff type. In particular, we construct some families which fail the integral Hasse principle via the Brauer--Manin obstruction with some counting results of similar nature to those in \cite{GS22}, \cite{LM20} and \cite{CTWX20}. We also give some counterexamples to strong approximation (where integral points can exist) which can be explained by the Brauer--Manin obstruction, and study a few aspects of rational points on affine surfaces.
\end{abstract}

\section{Introduction}
Let $X$ be an affine variety over $\Qbb$, and $\mathcal{X}$ an integral model of $X$ over $\Zbb$, i.e. an affine scheme of finite type over $\Zbb$ whose generic fiber is isomorphic to $U$. Define the set of adelic points $X(\textbf{\textup{A}}_{\Qbb}) := \sideset{}{'}\prod_{p} X(\Qpbb)$, where $p$ is a prime number or $p = \infty$ (with $\Qbb_{\infty} = \Rbb$). Similarly, define $\mathcal{X}(\textbf{A}_{\Zbb}) := \prod_{p} \mathcal{X}(\Zpbb)$ (with $\Zbb_{\infty} = \Rbb$). 

We say that $X$ \textit{fails the Hasse principle} (resp. $\mathcal{X}$ \textit{fails the integral Hasse principle}) if 
$$ X(\textbf{\textup{A}}_{\Qbb}) \not= \emptyset \hspace{0.5cm}\textup{but}\hspace{0.5cm} X(\Qbb) = \emptyset $$
$$ (\textup{resp.}\; \mathcal{X}(\textbf{A}_{\Zbb}) \not= \emptyset \hspace{0.5cm}\textup{but}\hspace{0.5cm} \mathcal{X}(\Zbb) = \emptyset.) $$
We say that $X$ \textit{satisfies weak approximation} if the image of $X(\Qbb)$ in $\prod_{v} X(\Qbb_{v})$ is dense, where the product is taken over all places of $\Qbb$. Finally, we say that $\mathcal{X}$ \textit{satisfies strong approximation} if $\mathcal{X}(\Zbb)$ is dense in $\mathcal{X}(\textbf{A}_{\Zbb})_{\bullet} := \prod_{p} \mathcal{X}(\Zpbb) \times \pi_{0}(X(\Rbb))$, where $\pi_{0}(X(\Rbb))$ denotes the set of connected components of $X(\Rbb)$. 

In general, few varieties satisfy the Hasse principle. In 1970, Manin introduced a natural cohomological obstruction to the Hasse principle, namely the \textbf{Brauer--Manin obstruction} (which has been extended to its integral version in \cite{CTX09}). If $\textup{Br}\,X$ denotes the cohomological Brauer group of $X$, i.e. $\textup{Br}\,X := \textup{H}^{2}_{\textup{ét}}(X,\Gmbb)$, we have a natural pairing from class field theory:
$$ X(\textbf{\textup{A}}_{\Qbb}) \times \textup{Br}\,X \rightarrow \Qbb/\Zbb. $$
If we define $X(\textbf{\textup{A}}_{\Qbb})^{\textup{Br}}$ to be the left kernel of this pairing, then the exact sequence of Albert--Brauer--Hasse--Noether gives us the relation:
$$ X(\Qbb) \subseteq X(\textbf{\textup{A}}_{\Qbb})^{\textup{Br}} \subseteq X(\textbf{\textup{A}}_{\Qbb}). $$
Similarly, by defining the Brauer--Manin set $\mathcal{X}(\textbf{A}_{\Zbb})^{\text{Br}}$, we also have that 
$$ \mathcal{X}(\Zbb) \subseteq \mathcal{X}(\textbf{A}_{\Zbb})^{\text{Br}} \subseteq \mathcal{X}(\textbf{A}_{\Zbb}). $$
This gives the so-called \textit{integral} Brauer--Manin obstruction. We say that \textit{the Brauer--Manin obstruction to the (resp. integral) Hasse principle is the only one} if 

$$ X(\textbf{A}_{\Qbb})^{\text{Br}} \not= \emptyset \iff X(\Qbb) \not= \emptyset.$$

$$ (\mathcal{X}(\textbf{A}_{\Zbb})^{\text{Br}} \not= \emptyset \iff \mathcal{X}(\Zbb) \not= \emptyset.)$$

We are interested particularly in the case where $X$ is a hypersurface, defined by a polynomial equation of degree $d$ in an affine space. The case $d = 1$ is easy and elementary. The case $d = 2$ considers the arithmetic of quadratic forms: for rational points, the Hasse principle is always satisfied by the Hasse--Minkowski theorem, and for integral points, the Brauer--Manin obstruction to the integral Hasse principle is the only one (up to an isotropy assumption) due to work of Colliot-Thélène, Xu \cite{CTX09} and Harari \cite{Ha08}.

However, the case $d = 3$ (of cubic hypersurfaces) is still largely open, especially for integral points. Overall, the arithmetic of integral points on the affine cubic surfaces over number fields is still little understood. An interesting example of affine cubic surfaces is given by \textbf{Markoff surfaces} $U_{m}$ which are defined by 
$$ x^{2} + y^{2} + z^{2} - xyz = m, $$ 
where $m$ is an integer parameter. In \cite{GS22}, Ghosh and Sarnak study the integral points on those affine Markoff surfaces $U_{m}$ both from a theoretical point of view and from numerical evidence. They prove that for almost all $m$, the integral Hasse principle holds, and that there are infinitely many $m$'s for which it fails. Furthermore, their numerical experiments suggest particularly a proportion of integers $m$ satisfying $|m| \leq M$ of the power $M^{0,8875\dots+o(1)}$ for which the integral Hasse principle is not satisfied.

Subsequently, Loughran and Mitankin \cite{LM20} proved that asymptotically only a proportion of $M^{1/2}/(\log M)^{1/2}$ of integers $m$ such that $-M \leq m \leq M$ presents an integral Brauer--Manin obstruction to the Hasse principle. They also obtained a lower bound, asymptotically $M^{1/2}/\log M$, for the number of integral Hasse failures which cannot be explained by the Brauer--Manin obstruction. After Colliot-Thélène, Wei, and Xu \cite{CTWX20} obtained a slightly stronger lower bound than the one given in \cite{LM20}, no better result than their number $M^{1/2}/(\log M)^{1/2}$ has been known until now. In other words, with all the current results, one does not have a satisfying comparison between the numbers of Hasse failures which can be explained by the Brauer--Manin obstruction and which cannot be explained by this obstruction. Meanwhile, for strong approximation, it has been proven to almost never hold for Markoff surfaces in \cite{LM20} and then absolutely never be the case in \cite{CTWX20}. 
Here we recall an important conjecture given by Ghosh and Sarnak.

\begin{conjecture}[Conjecture 10.2 in \cite{GS22}]
The number of integral Hasse failures satisfies that
$$ \# \{m \in \Zbb: 0 \leq m \leq M,\ \mathcal{U}_{m}(\Abb_{\Zbb}) \not= \emptyset\ \text{but}\ \mathcal{U}_{m}(\Zbb) = \emptyset \} \approx C_{0}M^{\theta}, $$
for some $C_{0} > 0$ and some $\frac{1}{2} < \theta < 1$.
\end{conjecture}

The above conjecture also means that \emph{almost all} counterexamples to the integral Hasse principle for Markoff surfaces cannot be explained by the Brauer--Manin obstruction, thanks to the result obtained by \cite{LM20}. 
\\~\\
\indent In recent work \cite{Dao22}, we study the set of integral points of a different Markoff-type cubic surfaces whose origin is similar to that of the original Markoff surfaces $U_{m}$, namely the \textit{relative character varieties}, using the Brauer--Manin obstruction as well. The surfaces are given by the cubic equation:

$$ x^{2} + y^{2} + z^{2} + xyz = ax + by + cz + d, $$

\noindent where $a,b,c,d \in \Zbb$ are parameters which satisfy some specific relations (see \cite{CL09}). Due to the similar appearance to the original Markoff surfaces, one may expect to find some similarities in their arithmetic properties. One of the main results in that paper is the following, saying that a positive proportion of these relative character varieties have no (algebraic) Brauer--Manin obstruction to the integral Hasse principle as well as fail strong approximation, and those failures can be explained by the Brauer--Manin obstruction.

\begin{theorem}
Let $\mathcal{U}$ be the affine scheme over $\Zbb$ defined by 
\begin{equation*}
x^{2} + y^{2} + z^{2} + xyz = ax + by + cz + d,
\end{equation*}
where 
\begin{equation*}
	\begin{cases*}
	a = k_{1}k_{2} + k_{3}k_{4} \\
	b = k_{1}k_{4} + k_{2}k_{3} \\
	c = k_{1}k_{3} + k_{2}k_{4}
	\end{cases*}
	\hspace{0.5cm} \textup{and} \hspace{0.5cm} d = 4 - \sum_{i=1}^{4} k_{i}^{2} - \prod_{i=1}^{4} k_{i}, 
\end{equation*}
such that its natural compactification $X$ is smooth over $\Qbb$. Then we have
$$ \#\{ k = (k_{1},k_{2},k_{3},k_{4}) \in \Zbb^{4}, |k_{i}| \leq M\; \forall\; 1 \leq i \leq 4 : \emptyset \not= \mathcal{U}(\textbf{\textup{A}}_{\Zbb})^{\textup{Br}_{1}} \not= \mathcal{U}(\textbf{\textup{A}}_{\Zbb}) \} \asymp M^{4} $$
as $M \rightarrow +\infty$.
\end{theorem}

Finally, in this paper, we are going to study certain other analogous varieties, in the world of K3 surfaces instead of cubic surfaces. Let $K$ be a number field. Let $X \subset \Pbb^{1} \times \Pbb^{1} \times \Pbb^{1}$ be a smooth K3 surface over $K$, given by a $(2, 2, 2)$ form
$$ F(X_{1}, X_{2} ; Y_{1}, Y_{2} ; Z_{1}, Z_{2}) \in K[X_{1}, X_{2} ; Y_{1}, Y_{2} ; Z_{1}, Z_{2}]. $$
Then $X$ is called a \emph{Wehler} K3 surface; in particular, $X$ is an elliptic surface whose projections $p_{i} : X \rightarrow \Pbb^{1}$ ($i \in \{1,2,3\}$) have fibers as curves of arithmetic genus $1$.

A \textbf{Markoff-type K3 surface} $W$ is a Wehler K3 surface whose $(2, 2, 2)$-form $F$ is invariant under the action of the group $\mathcal{G} \subset \text{Aut}(\Pbb^{1} \times \Pbb^{1} \times \Pbb^{1})$ generated by $(x,y,z) \mapsto (-x,-y,z)$ and permutations of $(x,y,z)$. By \cite{FLST22}, there exist $a, b, c, d, e \in k$ so that the $(2, 2, 2)$-form $F$ that defines $W$ has the affine form:
$$ ax^{2}y^{2}z^{2} + b(x^{2}y^{2} + y^{2}z^{2} + z^{2}x^{2}) + cxyz + d(x^{2} + y^{2} + z^{2}) + e = 0. $$

Our main results show the Brauer--Manin obstructions with respect to explicit elements of the algebraic Brauer groups for the existence of integral points on \textbf{three} concrete families of Markoff-type K3 surfaces (MK3 surfaces). One of them, as the most \emph{general} one, is the following.

\begin{theorem}
For $k \in \Zbb$, let $W_{k} \subset \Pbb^{1} \times \Pbb^{1} \times \Pbb^{1}$ be the MK3 surfaces defined over $\Qbb$ by the $(2,2,2)$-form 
\begin{equation}
    F_{3}(x,y,z) = x^{2} + y^{2} + z^{2} + 4(x^{2}y^{2} + y^{2}z^{2} + z^{2}x^{2}) - 16x^{2}y^{2}z^{2} - k = 0.
\end{equation}
Let $\mathcal{U}_{k}$ be the integral model of $U_{k}$ defined over $\Zbb$ by the same equation. If $k$ satisfies the conditions:
\begin{enumerate}
    \item $k = -\frac{1}{4}(1 + 27\ell^{2})$ where $\ell \in \Zbb$ such that $\ell \equiv \pm 1$ \textup{mod} $8$, $\ell \equiv 1$ \textup{mod} $5$, $\ell \equiv 3$ \textup{mod} $7$, and $\ell \not\equiv \pm 10$ \textup{mod} $37$;
    \item $p \equiv \pm 1$ \textup{mod} $24$ for any prime divisor $p$ of $\ell$,
\end{enumerate}
then there is an algebraic Brauer--Manin obstruction to the integral Hasse principle on $\mathcal{U}_{k}$ with respect to the subgroup $A \subset \textup{Br}_{1}\,U_{k}/\textup{Br}_{0}\,U_{k}$ generated by the elements $\mathcal{A}_{1} = (4x^{2}+1, -2(4k+1))$ and $\mathcal{A}_{2} = (4y^{2}+1, -2(4k+1))$, i.e., $\mathcal{U}_{k}(\Zbb) \subset \mathcal{U}_{k}(\textbf{\textup{A}}_{\Zbb})^{A} = \emptyset$.
\end{theorem}

Our final result deals with the counting problem on the number of counterexamples to the integral Hasse principle for Wehler K3 surfaces of Markoff type. Recall that for Markoff surfaces, Loughran and Mitankin \cite{LM20} proved that asymptotically only a proportion of $M^{1/2}/(\log M)^{1/2}$ of integers $m$ such that $|m| \leq M$ presents an integral Brauer--Manin obstruction.

\begin{theorem}
For the above three families of MK3 surfaces, we have
	$$ \# \{k \in \Zbb: |k| \leq M,\ \mathcal{U}_{k}(\textbf{\textup{A}}_{\Zbb}) \not= \emptyset,\ \mathcal{U}_{k}(\textbf{\textup{A}}_{\Zbb})^{\textup{Br}} = \emptyset \} \gg \frac{M^{1/2}}{\textup{log}\,M}, $$
as $M \rightarrow +\infty$.
\end{theorem}

The structure of the paper is as follows. In Section 2, we provide some background on Wehler K3 surfaces and a recent study of the Markoff-type K3 (MK3) surfaces, as well as introduce the \emph{three} explicit families of MK3 surfaces that interest us. In Section 3, we first discuss some geometry of Wehler K3 surfaces and their Brauer groups. After the general setting, we turn our attention to a particular family of MK3 surfaces, where we explicitly calculate the algebraic Brauer group of the projective surfaces, and then we complete the analysis of the Brauer group by calculating the algebraic Brauer group of the affine surfaces. In Section 4, we use the Brauer group to give explicit examples of Brauer--Manin obstructions to the integral Hasse principle for three families of MK3 surfaces, and give some counting results for the Hasse failures. Finally, in Section 5, we make some important remarks to compare the main results in this paper to those of Markoff-type cubic surfaces in recent works, and we also give some counterexamples to strong approximation which can be explained by the Brauer--Manin obstruction.
\\~\\
\indent \textbf{Notation.} Let $k$ be a field and $\overline{k}$ a separable closure of $k$. Let $G_{k} := \textup{Gal}(\overline{k}/k)$ be the absolute Galois group. A $k$-variety is a separated $k$-scheme of finite type. If $X$ is a $k$-variety, we write $\overline{X} = X \times_{k} \overline{k}$. Let $k[X] = \textup{H}^{0}(X,\mathcal{O}_{X})$ and $\overline{k}[X] = \textup{H}^{0}(X,\mathcal{O}_{\overline{X}})$. If $X$ is an integral $k$-variety, let $k(X)$ denote the function field of $X$. If $X$ is a geometrically integral $k$-variety, let $\overline{k}(X)$ denote the function field of $\overline{X}$. 

Let $\textup{Pic}\,X = \textup{H}^{1}_{\textup{Zar}}(X,\Gmbb) = \textup{H}^{1}_{\textup{ét}}(X,\Gmbb)$ denote the Picard group of a scheme $X$. Let $Br\,X = \textup{H}^{2}_{\textup{ét}}(X,\Gmbb)$ denote the Brauer group of $X$. Let
$$ \textup{Br}_{1}\,X := \textup{Ker}[\textup{Br}\,X \rightarrow \textup{Br}\,\overline{X}] $$
denote the \textbf{algebraic Brauer group} of a $k$-variety $X$ and let $\textup{Br}_{0}\,X \subset \textup{Br}_{1}\,X$ denote the image of $\textup{Br}\,k \rightarrow \textup{Br}\,X$. The image of $\textup{Br}\,X \rightarrow \textup{Br}\,\overline{X}$ is called the \textbf{transcendental Brauer group} of $X$.

Given a field $F$ of characteristic zero containing a primitive $n$-th root of unity $\zeta = \zeta_{n}$, we have $\textup{H}^{2}(F,\mu^{\otimes 2}) = \textup{H}^{2}(F,\mu_{n}) \otimes \mu_{n}$. The choice of $\zeta_{n}$ then defines an isomorphism $\textup{Br}(F)[n] = \textup{H}^{2}(F,\mu_{n}) \cong \textup{H}^{2}(F,\mu_{n}^{\otimes 2})$. Given two elements $f, g \in F^{\times}$, we have their classes $(f)$ and $(g)$ in $F^{\times}/F^{\times n} = \textup{H}^{1}(F,\mu_{n})$. We denote by $(f,g)_{\zeta} \in \textup{Br}(F)[n] = \textup{H}^{2}(F,\mu_{n})$ the class corresponding to the cup-product $(f) \cup (g) \in \textup{H}^{2}(F,\mu_{n}^{\otimes 2})$. Suppose $F/E$ is a finite Galois extension with Galois group $G$. Given $\sigma \in G$ and $f,g \in F^{\times}$, we have $\sigma((f,g)_{\zeta_{n}}) = (\sigma(f),\sigma(g))_{\sigma(\zeta_{n})} \in \textup{Br}(F)$. In particular, if $\zeta_{n} \in E$, then $\sigma((f,g)_{\zeta_{n}}) = (\sigma(f),\sigma(g))_{\zeta_{n}}$. For all the details, see Section 4.6, 4.7 in \cite{GS17}.

Let $R$ be a discrete valuation ring with fraction field $F$ and residue field $\kappa$. Let $v$ denote the valuation $F^{\times} \rightarrow \Zbb$. Let $n > 1$ be an integer invertible in $R$. Assume that $F$ contains a primitive $n$-th root of unity $\zeta$. For $f,g \in F^{\times}$, we have the residue map
$$ \partial_{R} : \textup{H}^{2}(F,\mu_{n}) \rightarrow \textup{H}^{1}(\kappa,\Zbb/n\Zbb) \cong \textup{H}^{1}(\kappa,\mu_{n}) = \kappa^{\times}/\kappa^{\times n}, $$
where $\textup{H}^{1}(\kappa,\Zbb/n\Zbb) \cong \textup{H}^{1}(\kappa,\mu_{n})$ is induced by the isomorphism $\Zbb/n\Zbb \simeq \mu_{n}$ sending $1$ to $\zeta$. This map sends the class of $(f,g)_{\zeta} \in \textup{Br}(F)[n] = \textup{H}^{2}(F,\mu_{n})$ to 
$$ (-1)^{v(f)v(g)} \textup{class}(g^{v(f)}/f^{v(g)}) \in \kappa/\kappa^{\times n}. $$

For a proof of these facts, see \cite{GS17}. Here we recall some precise references. Residues in Galois cohomology with finite coefficients are defined in \cite{GS17}, Construction 6.8.5. Comparison of residues in Milnor K-Theory and Galois cohomology is given in \cite{GS17}, Proposition 7.5.1. The explicit formula for the residue in Milnor’s group K2 of a discretely valued field is given in \cite{GS17}, Example 7.1.5.
\\~\\
\indent \textbf{Acknowledgements.} I thank Cyril Demarche for his help and supervision during my PhD study at the Institute of Mathematics of Jussieu. I thank Kevin Destagnol for his help with the computations in Section 4.3 using analytic number theory. I thank Vladimir Mitankin for his useful remarks and suggestions, especially regarding Section 5. I thank Jean-Louis Colliot-Thélène, Fei Xu, and Daniel Loughran for their nice comments and encouragements. This project has received funding from the European Union’s Horizon 2020 Research and Innovation Programme under the Marie Skłodowska-Curie grant agreement No. 754362.

\section{Background}
We give some notations and results about Wehler K3 surfaces and the so-called \emph{Markoff-type} K3 surfaces that we study in this paper.

\subsection{Wehler K3 surfaces}
Consider the variety $M = \Pbb^{1} \times \Pbb^{1} \times \Pbb^{1}$ and let $\pi_{1}$, $\pi_{2}$, and $\pi_{3}$ be the projections on the first, second, and third factor: $\pi_{i}(z_{1},z_{2},z_{3}) = z_{i}$. Denote by $L_{i}$ the line bundle $\pi_{i}^{*}(\mathcal{O}(1))$ and set
$$ L = L_{1}^{2} \otimes L_{2}^{2} \otimes L_{3}^{2} = \pi_{1}^{*}(\mathcal{O}(2)) \otimes \pi_{2}^{*}(\mathcal{O}(2)) \otimes \pi_{3}^{*}(\mathcal{O}(2)). $$
Since $K_{\Pbb^{1}} = \mathcal{O}(-2)$, this line bundle $L$ is the dual of the canonical bundle $K_{M}$. By definition, $|L| \simeq \Pbb(\textup{H}^{0}(M,L))$ is the linear system of surfaces $W \subset M$ given by the zeroes of global sections $P \in \textup{H}^{0}(M,L)$. Using affine coordinates $(x_{1}, x_{2}, x_{3})$ on $M = \Pbb^{1} \times \Pbb^{1} \times \Pbb^{1}$, such a surface is defined by a polynomial equation $F(x_{1},x_{2},x_{3}) = 0$ whose degree with respect to each variable is $\leq 2$. These surfaces will be referred to as \textbf{Wehler surfaces}; modulo $\textup{Aut}(M)$, they form a family of dimension $17$.

Fix $k \in \{1,2,3\}$ and denote by $i < j$ the other indices. If we project $W$ to $\Pbb^{1} \times \Pbb^{1}$ by $\pi_{ij} = (\pi_{i}, \pi_{j})$, we get a $2$ to $1$ cover (the generic fiber is made of two points, but some fibers may be rational curves). As soon as $W$ is \emph{smooth}, the involution $\sigma_{k}$ that permutes the two points in each (general) fiber of $\pi_{ij}$ is an involutive automorphism of $W$; indeed $W$ is a K3 surface and any birational self-map of such a surface is an automorphism (see \cite{Bi97}, Lemma 1.2). By \cite{CD22}, Proposition 3.1, we have the following general result.

\begin{proposition} 
There is a countable union of proper Zariski closed subsets $(S_{i})_{i \geq 0}$ in $|L|$ such that:
\begin{enumerate}
\item[(1)] If $W$ is an element of $|L| \setminus S_{0}$, then $W$ is a smooth K3 surface and $W$ does not contain any fiber of the projections $\pi_{ij}$, i.e., each of the three projections $(\pi_{ij})_{|W} : W \rightarrow \Pbb^{1} \times \Pbb^{1}$ is a finite map;

\item[(2)] If $W$ is an element of $|L| \setminus (\cup_{i \geq 0} S_{i})$, the restriction morphism $\textup{Pic}\,M \rightarrow \textup{Pic}\,W$ is surjective. In particular, the Picard number of $W$ is equal to $3$.
\end{enumerate}
\end{proposition}

From the second assertion, we deduce that for a \emph{very general} $W$, $\textup{Pic}\,W$ is isomorphic to $\textup{Pic}\,M$: it is the free Abelian group of rank $3$, generated by the classes
$$ D_{i} := [(L_{i})_{|W}]. $$
The elements of $|(L_{i})_{|W}|$ are the curves of $W$ given by the equations $z_{i} = \alpha$ for some $\alpha \in \Pbb^{1}$. The arithmetic genus of these curves is equal to $1$: in other words, the projection $(\pi_{i})_{|W} : W \rightarrow \Pbb^{1}$ is a genus $1$ fibration (see \cite{Bi97}, Lemma 1.1). Moreover, for a general choice of $W$ in $|L|$, $(\pi_{i})_{|W}$ has 24 singular fibers of type $\textup{I}_{1}$, i.e. isomorphic to a rational curve with exactly one simple double point. The intersection form is given by $D_{i}^{2} = 0$ and $(D_{i}.D_{j}) = 2$ if $i \not= j$, so that its matrix is given by
\[
\begin{pmatrix}
0 & 2 & 2\\
2 & 0 & 2\\
2 & 2 & 0
\end{pmatrix}.
\]

By \cite{Bi97}, Proposition 1.5 or \cite{CD22}, Lemma 3.2, we have the following result about the actions of the subgroup of $\textup{Aut}(W)$ generated by $\sigma_{1}, \sigma_{2}, \sigma_{3}$ on the geometry of $W$. 

\begin{proposition}
Assume that $W$ does not contain any fiber of the projection $\pi_{ij}$. Then the involution $\sigma^{*}_{k}$ preserves the subspace $\Zbb D_{1} \oplus \Zbb D_{2} \oplus \Zbb D_{3}$ of $\textup{NS}\,W$ and
$$ \sigma^{*}_{k}(D_{i}) = D_{i}, \hspace{0.25cm} \sigma^{*}_{k}(D_{j}) = D_{j}, \hspace{0.25cm} \sigma^{*}_{k}(D_{k}) = -D_{k} + 2D_{i} + 2D_{j}. $$
In other words, the matrices of the $\sigma^{*}_{i}$ in the basis $(D_{1}, D_{2}, D_{3})$ are:
\[
\sigma^{*}_{1} = \begin{pmatrix}
-1 & 0 & 0\\
2 & 1 & 0\\
2 & 0 & 1
\end{pmatrix}, \; \sigma^{*}_{2} = \begin{pmatrix}
1 & 2 & 0\\
0 & -1 & 0\\
0 & 2 & 1
\end{pmatrix}, \; \sigma^{*}_{3} = \begin{pmatrix}
1 & 0 & 2\\
0 & 1 & 2\\
0 & 0 & -1
\end{pmatrix}.
\]
\end{proposition}

Combining these two propositions, we have the following (see \cite{Bi97}, Proposition 1.3 or \cite{CD22}, Proposition 3.3):

\begin{proposition}
If $W$ is a very general Wehler surface then:
\begin{enumerate}
\item[(1)] $W$ is a smooth K3 surface with Picard number $3$;

\item[(2)] $\textup{Aut}(W) = \langle \sigma_{1}, \sigma_{2}, \sigma_{3} \rangle$, which is a free product of three copies of $\Zbb/2\Zbb$, and $\textup{Aut}(W)^{*}$ is a finite index subgroup in the group of integral isometries of $\textup{NS}\,W$.
\end{enumerate}
\end{proposition}

Besides the three involutions $\sigma_{1}, \sigma_{2}, \sigma_{3}$, depending on the symmetries of the defining polynomial $F$, the automorphism group of a Wehler surface $W$ may contain additional automorphisms. Typical examples include symmetry in $x, y, z$ that allows permutation of the coordinates, and power symmetry that allows the signs of two of $x, y, z$ to be reversed. For example, the original Markoff equation permits these extra automorphisms; and hereafter we consider analogous Markoff-type surfaces. Note that all the above results are true for \emph{very general} Wehler surfaces; as we will see, our examples of surfaces to study in this paper are in fact very far from being general, which leads to many different results in the end.

\subsection{Markoff-type K3 surfaces}
Now let $K$ be a field. A Wehler surface $W$ over $K$ is then a surface 
$$ W = \{\overline{F} = 0\} \subset \Pbb^{1} \times \Pbb^{1} \times \Pbb^{1} $$
defined by a $(2,2,2)$-form 
$$ \overline{F}(x,r;y,s;z,t) \in K[x,r;y,s;z,t]. $$
Using the affine coordinates $(x,y,z)$, we let 
$$ F(x,y,z) = \overline{F}(x,1;y,1;z,1), $$
and then $W$ is the closure in $\Pbb^{1} \times \Pbb^{1} \times \Pbb^{1}$ of the affine surface, which by abuse of notation we also denote by 
$$ W : F(x,y,z) = 0. $$

We say that $W$ is \emph{non-degenerate} if it satisfies the following two conditions:
\begin{enumerate}
\item[(i)] The projection maps $\pi_{12}, \pi_{13}, \pi_{23}$ are finite.
\item[(ii)] The generic fibers of the projection maps $\pi_{1}, \pi_{2}, \pi_{3}$ are smooth curves, in which case the smooth fibers are necessarily curves of genus $1$, since they are $(2, 2)$ curves in $\Pbb^{1} \times \Pbb^{1}$.
\end{enumerate}

By analogy with the classical Markoff equation, we say that $W$ is of \emph{Markoff type} (MK3) if it is symmetric in its three coordinates and invariant under double sign changes. An MK3 surface admits a group of automorphisms $\Gamma$ generated by the three involutions, coordinate permutations, and sign changes. Following the notations in \cite{FLST22}, we define:

\begin{definition}
We let $\mathfrak{S}_{3}$, the symmetric group on $3$ letters, act on $\Pbb^{1} \times \Pbb^{1} \times \Pbb^{1}$ by permuting the coordinates, and we let the group
$$ (\mu_{2}^{3})_{1} := \{(\alpha, \beta, \gamma) : \alpha, \beta, \gamma \in \mu_{2}\, \textup{and}\, \alpha\beta\gamma = 1\} $$
act on $\Pbb^{1} \times \Pbb^{1} \times \Pbb^{1}$ via sign changes,
$$ (\alpha, \beta, \gamma)(x, y, z) = (\alpha x, \beta y, \gamma z). $$
In this way we obtain an embedding
$$ \mathcal{G} := (\mu_{2}^{3})_{1} \rtimes \mathfrak{S}_{3} \hookrightarrow \textup{Aut}(\Pbb^{1} \times \Pbb^{1} \times \Pbb^{1}). $$
\end{definition}

\begin{definition}
A \emph{Markoff-type K3} (MK3) surface $W$ is a Wehler surface whose $(2, 2, 2)$-form $F(x,y,z)$ is invariant under the action of $\mathcal{G}$, i.e., the $(2, 2, 2)$-form $F$ defining $W$ satisfies
\[
\begin{aligned}
F(x, y, z) &= F(-x, -y, z) = F(-x, y, -z) = F(x, -y, -z),\\
F(x, y, z) &= F(z, x, y) = F(y, z, x) = F(x, z, y) = F(y, x, z) = F(z, y, x).
\end{aligned}
\]
\end{definition}

By \cite{FLST22}, Proposition 7.5, we have the following key result about the defining form of MK3 surfaces.

\begin{proposition}
Let $W/K$ be a (possibly degenerate) MK3 surface.
\begin{enumerate}
\item[(a)] There exist $a,b,c,d,e \in K$ so that the $(2, 2, 2)$-form $F$ that defines $W$ has the form
\begin{equation}
F(x, y, z) = ax^{2}y^{2}z^{2} + b(x^{2}y^{2} + x^{2}z^{2} + y^{2}z^{2}) + cxyz + d(x^{2} + y^{2} + z^{2}) + e = 0.
\end{equation}

\item[(b)] Let $F$ be as in (a). Then $W$ is a non-degenerate, i.e., the projections $\pi_{ij} : W \rightarrow \Pbb^{1} \times \Pbb^{1}$ are \emph{quasi-finite}, if and only if
$$ c \not= 0, \hspace{0.5cm} be \not= d^{2}, \hspace{0.5cm} \textup{and} \hspace{0.5cm} ad \not= b^{2}. $$
\end{enumerate}
\end{proposition}

\begin{remark}
We can recover the original Markoff equation for a surface $S_{k}$ as a special case of a form $F$ with $a=b=0, c=-1, d=1, e=-k$. More precisely, $S_{k}$ is given by the affine equation
$$ F(x,y,z) = x^{2} + y^{2} + z^{2} - xyz - k = 0. $$
We note, however, that the Markoff equation is degenerate, despite the involutions being well-defined on the affine Markoff surface $S_{k}$. This occurs because the involutions are not well-defined at some of the points at infinity in the closure of $S_{k}$ in $\Pbb^{1} \times \Pbb^{1} \times \Pbb^{1}$; for example, the inverse image $\pi_{12}^{-1}([1:0], [1:0])$ in $X_{k}$ is a line isomorphic to $\Pbb^{1}$.
\end{remark}

Now we are ready to introduce the \textbf{three} families of MK3 surfaces that we study in this paper. For $k \in \Zbb$, let $W_{k} \subset \Pbb^{1} \times \Pbb^{1} \times \Pbb^{1}$ be the MK3 surface defined over $\Qbb$ by one of the following $(2,2,2)$-forms:

\begin{equation}
F_{1}(x,y,z) = x^{2} + y^{2} + z^{2} - 4x^{2}y^{2}z^{2} - k = 0;
\end{equation}

\begin{equation}
F_{2}(x,y,z) = x^{2} + y^{2} + z^{2} - 4(x^{2}y^{2} + y^{2}z^{2} + z^{2}x^{2}) + 16x^{2}y^{2}z^{2} - k = 0;
\end{equation}

\begin{equation}
F_{3}(x,y,z) = x^{2} + y^{2} + z^{2} + 4(x^{2}y^{2} + y^{2}z^{2} + z^{2}x^{2}) - 16x^{2}y^{2}z^{2} - k = 0.
\end{equation}

It is important to note that all these families of Markoff-type K3 surfaces are \emph{degenerate} in the sense that every member of each family contains a fiber (a line isomorphic to $Pbb^{1}$) of the projection $\pi_{ij}$. Furthermore, there exist $\Qbb$-rational points at infinity on every member of each family of Markoff-type K3 surfaces considered above:
\begin{equation*}
\begin{cases}
([1:0], [1:1], [1:2]) \in \{\overline{F_{1}} = 0\};\\

([1:0], [1:0], [1:2]) \in \{\overline{F_{2}} = 0\};\\

([1:0], [1:0], [1:2]) \in \{\overline{F_{3}} = 0\}.
\end{cases}
\end{equation*}

In this paper, we study some explicit cases when there are however no integral points due to the Brauer--Manin obstruction.

\section{The Brauer group of Markoff-type K3 surfaces}
We are particularly interested in the geometry of the \textbf{third} Markoff-type K3 surfaces defined by $(5)$, as they are more complicated and general than the other two. In addition, under our specific conditions, the first and second surfaces are always \emph{singular} at infinity (for example, at the points $([1:0],[0:1],[1:0])$ and $([1:0],[1:2],[1:2])$, respectively), but the third one is \emph{smooth}. Before studying the arithmetic problem of integral points, we will give some explicit computations on the (geometric) Picard group and the (algebraic) Brauer group of these surfaces. Recall that by \cite{Bi97}, Proposition 1.3 or {\cite{CD20}, Proposition 3.3, for a \emph{very general} $W$, $\textup{Pic}\,\overline{W}$ is isomorphic to $\textup{Pic}(\Pbb^{1} \times \Pbb^{1} \times \Pbb^{1})$, i.e. $\textup{Pic}\,\overline{W}$ is generated by the classes $D_{i}$ so the Picard number of $\overline{W}$ equals $3$. However, as previously discussed, we will see in this section that our example of MK3 surfaces is very \emph{special}.

\subsection{Geometry of K3 surfaces}
Let $k$ be a number field. Let $W \subset \Pbb^{1} \times \Pbb^{1} \times \Pbb^{1}$ be a smooth Wehler surface over $k$ defined by a $(2,2,2)$-form $F = 0$. For distince $i,j \in \{1,2,3\}$, we keep the notations $\pi_{i} : W \rightarrow \Pbb^{1}$ and $\pi_{ij} : W \rightarrow \Pbb^{1} \times \Pbb^{1}$ of the various projections of $W$ onto one or two copies of $\Pbb^{1}$. Let $D_{i}$ denote the divisor class represented by a fiber of $\pi_{i}$. We find that $(D_{i}.D_{j}) = 2$ for $i \not= j$ and since any two different fibers of $\pi_{i}$ are disjoint, we have $D_{i}^{2} = 0$. It follows that the intersection matrix $((D_{i}.D_{j}))_{i,j}$ has rank $3$, so the $D_{i}$ generates a subgroup of rank $3$ of the N\'eron--Severi group $\textup{NS}\,\overline{W}$.

We have the following result for the Picard group of Wehler surfaces over an algebraically closed field $\overline{k}$.
\begin{proposition}
	Let $W \subset \Pbb^{1} \times \Pbb^{1} \times \Pbb^{1}$ be a smooth, projective, geometrically integral Wehler surface over $k$. Suppose the three planes at infinity $\{ rst = 0 \}$ cut out on $\overline{W}$ three distinct fibers $D_{1}, D_{2}, D_{3}$ over $\overline{k}$. Let $U \subset W$ be the complement of these fibers. Then $\overline{k}^{\times} = \overline{k}[U]^{\times}$ and the natural sequence
	$$ 0 \longrightarrow \bigoplus_{i=1}^{3} \Zbb D_{i} \longrightarrow \textup{Pic}\,\overline{W} \longrightarrow \textup{Pic}\,\overline{U} \longrightarrow 0 $$
	is exact.
\end{proposition}

\begin{proof}
To show that the above sequence is exact, it suffices to prove that the second arrow is an injective homomorphism. Let 
$$ a_{1}D_{1} + a_{2}D_{2} + a_{3}D_{3} = 0 \in \textup{Pic}\,\overline{W} $$
with $a,b,c \in \Zbb$. By the assumption that $(D_{i}.D_{i}) = 0$ and $(D_{i}.D_{j}) = 2$ for $1 \leq i \not= j \leq 3$, one has 
$$ 2a_{2} + 2a_{3} = 2a_{1} + 2a_{3} = 2a_{1} + 2a_{2} = 0, $$
so $a_{1} = a_{2} = a_{3} = 0$. In other words, this is another proof of the fact that $D_{1}, D_{2}, D_{3}$ are linearly independent in $\textup{Pic}\,\overline{W}$ and it also shows that $\overline{k}^{\times} = \overline{k}[U]^{\times}$ as desired.
\end{proof}

Now let $k$ be an arbitrary field. Recall that for a variety $X$ over $k$ there is a natural filtration on the Brauer group
$$ \textup{Br}_{0}\,X \subset \textup{Br}_{1}\,X \subset \textup{Br}\,X $$ which is defined as
$$ \textup{Br}_{0}\,X = \textup{Im}[\textup{Br}\,k \rightarrow \textup{Br}\,X], \hspace{0.5cm} \textup{Br}_{1}\,X = \textup{Ker}[\textup{Br}\,X \rightarrow \textup{Br}\,\overline{X}]. $$
From the Hochschild--Serre spectral sequence, we have the following spectral sequence:
\begin{equation}
E_{2}^{pq} = \textup{H}^{p}_{\textup{\'et}}(k, \textup{H}_{\textup{\'et}}^{q}(\overline{X},\Gmbb)) \Longrightarrow \textup{H}^{p+q}_{\textup{\'et}}(X,\Gmbb),
\end{equation}
which is contravariantly functorial in the $k$-variety $X$. It gives rise to the functorial exact sequence of terms of low degree:
\begin{equation}
\begin{aligned}
    0 &\longrightarrow \textup{H}^{1}(k,\overline{k}[X]^{\times}) \longrightarrow \textup{Pic}\,X \longrightarrow \textup{Pic}\,\overline{X}^{G_{k}} \longrightarrow \textup{H}^{2}(k, \overline{k}[X]^{\times}) \longrightarrow \textup{Br}_{1}\,X\\
    &\longrightarrow \textup{H}^{1}(k, \textup{Pic}\,\overline{X}) \longrightarrow \textup{Ker}[\textup{H}^{3}(k, \overline{k}[X]^{\times}) \rightarrow \textup{H}^{3}_{\textup{\'et}}(X,\Gmbb)]. 
\end{aligned}
\end{equation}
Let $X$ be a variety over a field $k$ such that $\overline{k}[X]^{\times} = \overline{k}^{\times}$. By Hilbert’s theorem 90 we have $\textup{H}^{1}(k, \overline{k}^{\times}) = 0$, then by the above sequence there is an exact sequence
\begin{equation}
\begin{aligned}
    0 &\longrightarrow \textup{Pic}\,X \longrightarrow \textup{Pic}\,\overline{X}^{G_{k}} \longrightarrow \textup{Br}\,k \longrightarrow \textup{Br}_{1}\,X \\
	&\longrightarrow \textup{H}^{1}(k, \textup{Pic}\,\overline{X}) \longrightarrow \textup{Ker}[\textup{H}^{3}(k, \overline{k}^{\times}) \rightarrow \textup{H}^{3}_{\textup{\'et}}(X,\Gmbb)].
\end{aligned}
\end{equation}
This sequence is also contravariantly functorial in $X$.

\begin{remark}
Let $X$ be a variety over a field $k$ such that $\overline{k}[X]^{\times} = \overline{k}^{\times}$. This assumption $\overline{k}[X]^{\times} = \overline{k}^{\times}$ holds for any proper, geometrically connected and geometrically reduced $k$-variety $X$.
\begin{enumerate} 
\item[(1)] If $X$ has a $k$-point, which defined a section of the structure morphism $X \rightarrow \textup{Spec}\,k$, then each of the maps $\textup{Br}\,k \longrightarrow \textup{Br}_{1}\,X$ and $\textup{H}^{3}(k, \overline{k}^{\times}) \rightarrow \textup{H}^{3}_{\textup{\'et}} (X, \Gmbb)$ has a retraction, hence is injective. (Then $\textup{Pic}\,X \longrightarrow \textup{Pic}\,\overline{X}^{G_{k}}$ is an isomorphism.) Therefore, we have an isomorphism
$$ \textup{Br}_{1}\,X/\textup{Br}\,k \cong \textup{H}^{1}(k, \textup{Pic}\,\overline{X}). $$ 

\item[(2)] If $k$ is a number field, then $\textup{H}^{3}(k, \overline{k}^{\times}) = 0$ (see \cite{CF67}, Chapter VII, Section 11.4, p. 199). Thus for a variety $X$ over a number field $k$ such that $\overline{k}[X]^{\times} = \overline{k}^{\times}$, we have an isomorphism 
$$ \textup{Br}_{1}\,X/\textup{Br}_{0}\,X \cong \textup{H}^{1}(k, \textup{Pic}\,\overline{X}). $$
\end{enumerate}
\end{remark}

If $X$ is a K3 surface, or more generally, $X$ is a smooth, projective and geometrically integral $k$-variety such that $\textup{H}^{1}(X,\mathcal{O}_{X}) = 0$, then the Picard group $\textup{Pic}\,\overline{X}$ and the N\'eron--Severi group $\textup{NS}\,\overline{X}$ are equal (see \cite{CTS21}, Corollary 5.1.3). Therefore, we have the following result (see \cite{CTS21}, Theorem 5.5.1). 

\begin{theorem}
Let $X$ be a smooth, projective and geometrically integral variety over a field $k$. Assume that $\textup{H}^{1}(X,\mathcal{O}_{X}) = 0$ and $\textup{NS}\,\overline{X}$ is torsion-free. Then $\textup{H}^{1}(k, \textup{Pic}\,\overline{X})$ and $\textup{Br}_{1}\,X/\textup{Br}_{0}\,X$ are finite groups.
\end{theorem}

The assumption of the above theorem is always true if $X$ is a K3 surface. Furthermore, by Skorobogatov and Zarhin, we have a stronger result for the Brauer group of K3 surfaces (see \cite{CTS21}, Theorem 16.7.2 and Collorary 16.7.3).

\begin{theorem}
Let $X$ be a K3 surface over a field $k$ finitely generated over $\Qbb$. Then $(\textup{Br}\,\overline{X})^{\Gamma}$ is finite. Moreover, the group $\textup{Br}\,X/\textup{Br}_{0}\,X$ is finite.
\end{theorem}

Next, we will give an explicit computation of the geometric Picard group and the algebraic Brauer group for the family of Markoff-type K3 surfaces defined by $(5)$.

\subsection{The geometric Picard group}
Using the explicit equations, we compute the geometric Picard group of the Markoff-type K3 surfaces in question. To bound the Picard number we use the method described in \cite{vL07b}. Let $X$ be any smooth surface over a number field $K$ and let $\mathfrak{p}$ be a prime of good reduction with residue field $k$. Let $\mathcal{X}$ be an integral model for $X$ over the localization $\mathcal{O}_{\mathfrak{p}}$ of the ring of integers $\mathcal{O}$ of $K$ at $\mathfrak{p}$ for which the reduction is smooth. Let $k'$ be any extension field of $k$. Then by abuse of notation, we will write $X_{k'}$ for $X \times_{\textup{Spec}\,\mathcal{O}_{\mathfrak{p}}} \,\textup{Spec}\,k'$. We need the following important result which describes the behavior of the Néron--Severi group under good reduction.

\begin{proposition}
Let $X$ be a smooth surface over a number field $K$ and let $\mathfrak{p}$ be a prime of good reduction with residue field $k$. Let $l$ be a prime not dividing $q = \# k$. Let $F$ denote the automorphism on $\textup{H}^{2}_{\textup{ét}}(X_{\overline{k}}, \Qbb_{l}(1))$ induced by $q$-th power Frobenius. Then there are natural injections
$$ \textup{NS}(X_{\overline{K}}) \otimes \Qbb_{l} \hookrightarrow \textup{NS}(X_{\overline{k}}) \otimes \Qbb_{l} \hookrightarrow \textup{H}^{2}_{\textup{ét}}(X_{\overline{k}}, \Qbb_{l})(1), $$
that respect the intersection pairing and the action of Frobenius respectively. The rank of $\textup{NS}(X_{\overline{k}})$ is at most the number of eigenvalues of $F$ that are roots of unity, counted with multiplicity.
\end{proposition}

\begin{proof}
See \cite{vL07a}, Proposition 6.2 and Corollary 6.4; or \cite{BL07}, Proposition 2.3.
\end{proof}

Recall that if $X$ is a K3 surface, then linear, algebraic, and numerical equivalence all coincide. This means that the Picard group $\textup{Pic}\,\overline{X}$ and the Néron--Severi group $\textup{NS}\,\overline{X}$ of $\overline{X} := X_{\overline{\Qbb}}$ are naturally isomorphic, finitely generated, and free. Their rank is called the \emph{geometric} Picard number of $X$ or the Picard number of $\overline{X}$. By the \emph{Hodge Index Theorem}, the intersection pairing on $\textup{Pic}\,\overline{W}$ is even, non-degenerate, and of signature $(1, \textup{rk}\,\textup{NS}\,\overline{W} - 1)$. 

\begin{proposition}
Let $W \subset \Pbb^{1} \times \Pbb^{1} \times \Pbb^{1}$ be a surface defined over $\Qbb$ by the $(2,2,2)$-form
$$ F(x,y,z) = x^{2}+y^{2}+z^{2} + 4(x^{2}y^{2}+y^{2}z^{2}+z^{2}x^{2}) - 16x^{2}y^{2}z^{2} - k = 0, $$
where $k \in \Zbb$. Consider the field extension $K := \Qbb(\sqrt{-1}, \sqrt{\alpha},\sqrt{\bar{\alpha}})$ where $\Delta = \frac{(4k-5)^{2}-32}{64}$, $\alpha = \frac{1}{2}\left(\frac{4k-1}{8} + \sqrt{\Delta}\right)$ and $\bar{\alpha} = \frac{1}{2}\left(\frac{4k-1}{8} - \sqrt{\Delta}\right)$. If $k$ satisfies the following conditions:
\begin{enumerate}
    \item None of $2(4k+1), \Delta, 2(4k+1)\Delta$ is a square in $\Qbb$;
    \item $k \equiv 3$ \textup{mod} $5$,
\end{enumerate}
such that $G := \textup{Gal}(K/\Qbb) \cong D_{4} \times \Zbb/2\Zbb$, then $W$ is a smooth K3 surface and the Picard number of $\overline{W} = W_{\overline{\Qbb}}$ equals $18$.
\end{proposition}

\begin{proof}
Since the surface $W \subset \Pbb^{1} \times \Pbb^{1} \times \Pbb^{1}$ is defined over $\Qbb$ by a $(2,2,2)$-form $F = 0$ with $k(4k+1)((4k-5)^{2}-32) \not= 0$, it is clear that $W$ is a smooth K3 surface. For $i = 1, 2, 3$, let $\pi_{i} : W \rightarrow \Pbb^{1}$ be the projection from $W$ to the $i$-th copy of $\Pbb^{1}$ in $\Pbb^{1} \times \Pbb^{1} \times \Pbb^{1}$. Let $D_{i}$ denote the divisor class represented by a \emph{smooth} fiber of $\pi_{i}$. By considering all the smooth fibers and the \emph{singular} fibers, the corresponding divisor classes on $\overline{W}$ are given explicitly as follows (denote by $[x:r], [y:s], [z:t]$ the coordinates for each point in $\Pbb^{1} \times \Pbb^{1} \times \Pbb^{1}$):
\begin{equation*}
    \begin{cases}
    D_{1}: [x:r] = [1:0], s^{2}t^{2} + 4y^{2}t^{2} + 4z^{2}s^{2} - 16y^{2}z^{2} = 0,\\
    D_{2}: [y:s] = [1:0], r^{2}t^{2} + 4x^{2}t^{2} + 4z^{2}r^{2} - 16x^{2}z^{2} = 0,\\
    D_{3}: [z:t] = [1:0], r^{2}s^{2} + 4x^{2}s^{2} + 4y^{2}r^{2} - 16x^{2}y^{2} = 0; 
    \end{cases}
\end{equation*}

\begin{equation*}
	\begin{cases}
    A_{1}: [x:r] = [\pm \sqrt{k}:1], (4k+1)y^{2}t^{2} + (4k+1)z^{2}s^{2} - (16k-4)y^{2}z^{2} = 0,\\
    A_{2}: [y:s] = [\pm \sqrt{k}:1], (4k+1)x^{2}t^{2} + (4k+1)z^{2}r^{2} - (16k-4)x^{2}z^{2} = 0,\\
    A_{3}: [z:t] = [\pm \sqrt{k}:1], (4k+1)x^{2}s^{2} + (4k+1)y^{2}r^{2} - (16k-4)x^{2}y^{2} = 0; 
    \end{cases}
\end{equation*}

\begin{equation*}
	\begin{cases}
    B_{1}: [x:r] = [\pm \frac{1}{2}:1], y^{2}t^{2} + z^{2}s^{2} - \frac{4k-1}{8}s^{2}t^{2} = 0,\\
    B_{2}: [y:s] = [\pm \frac{1}{2}:1], x^{2}t^{2} + z^{2}r^{2} - \frac{4k-1}{8}r^{2}t^{2} = 0,\\
    B_{3}: [z:t] = [\pm \frac{1}{2}:1], x^{2}s^{2} + y^{2}r^{2} - \frac{4k-1}{8}r^{2}s^{2} = 0; 
    \end{cases}
\end{equation*}

\begin{equation*}
    \begin{cases}
    C_{1}^{\pm \pm}: [x:r] = [\pm \sqrt{\frac{-1}{4}}:1], yz = \pm \sqrt{\frac{4k+1}{32}}st,\\
    C_{2}^{\pm \pm}: [y:s] = [\pm \sqrt{\frac{-1}{4}}:1], xz = \pm \sqrt{\frac{4k+1}{32}}rt,\\
    C_{3}^{\pm \pm}: [z:t] = [\pm \sqrt{\frac{-1}{4}}:1], xy = \pm \sqrt{\frac{4k+1}{32}}rs; 
    \end{cases}
\end{equation*}

and for $1 \leq i \not=j \leq 3$,
\begin{enumerate}
	\item $\ell_{ij}^{\pm \pm}: [x_{i}:r_{i}] = [\pm\sqrt{\alpha}:1], [x_{j}:r_{j}] = [\pm\sqrt{\bar{\alpha}}:1]$,

	\item $\overline{\ell_{ij}^{\pm \pm}}: [x_{i}:r_{i}] = [\pm\sqrt{\bar{\alpha}}:1], [x_{j}:r_{j}] = [\pm\sqrt{\alpha}:1]$,
\end{enumerate}
where $[x_{1}:r_{1}], [x_{2}:r_{2}], [x_{3}:r_{3}]$ denote $[x:r], [y:s], [z:t]$ respectively, while $(\pm \sqrt{\alpha},\pm \sqrt{\bar{\alpha}})$ are the solutions of the polynomial system 
\begin{equation*}
    \begin{cases}
    1 + 4a^{2} + 4b^{2} - 16a^{2}b^{2} = 0, \\
    a^{2} + b^{2} + 4a^{2}b^{2} - k = 0; \\
    \end{cases}
\end{equation*} 
i.e., they are deduced from the solutions of the polynomial equation $$ T^4 - \frac{4k-1}{8}T^2 + \frac{4k+1}{32} = 0, $$ where $\alpha = \frac{1}{2}\left(\frac{4k-1}{8} + \sqrt{\Delta}\right), \bar{\alpha} = \frac{1}{2}\left(\frac{4k-1}{8} - \sqrt{\Delta}\right)$ and $\Delta = \frac{(4k-5)^{2}-32}{64}$ is the discriminant of the associated quadratic polynomial.

We will now find explicit generators for the geometric Picard group of $W$. It is clear that $W$ is a K3 surface admitting an elliptic fibration $\pi_{1} : W \rightarrow \Pbb^{1}$ with a zero section defined by $\ell_{23}^{\++} \simeq \Pbb^{1}$. The N\'eron--Severi group of an elliptic fibration on the K3 surface is the lattice generated by the class of a (smooth) fiber, the class of the zero section, the classes of the irreducible components of the reducible fibers which do not intersect the zero section, and the Mordell--Weil group (the set of the sections). Following this property, we find a set of $18$ linearly independent divisor classes consisting of:
\begin{enumerate}
	\item[(i)] $D_{1}$ (a smooth fiber), $\ell_{23}^{++}$ (a zero section),
	\item[(ii)] $\{\ell_{12}^{++}, \ell_{12}^{+-}, \ell_{13}^{+-}, \ell_{12}^{--}, \ell_{12}^{-+}, \ell_{13}^{--}, \overline{\ell_{12}^{++}}, \overline{\ell_{12}^{+-}}, \overline{\ell_{13}^{+-}}, \overline{\ell_{12}^{--}}, \overline{\ell_{12}^{-+}}, \overline{\ell_{13}^{--}}\}$ (the classes of singular fibers not intersecting the zero section),
	\item[(iii)] $\{\overline{\ell_{23}^{++}}, \ell_{23}^{+-}, C_{2}^{+-}, C_{3}^{+-}\}$ (the set of some other sections).
\end{enumerate}
Their Gram matrix of the intersection pairing on $\textup{Pic}\,\overline{W}$ has determinant $-192$, which is nonzero, so they are indeed linearly independent as the intersection pairing is \emph{non-degenerate}. However, after considering the other divisor classes and their linear relations with this set of classes, we are able to find and work with another lattice of $18$ classes for technical reasons, such as symmetry for the (general) smooth fibers and the Gram determinant of smaller absolute value (in fact, the former lattice is a sublattice). More precisely, the intersection matrix associated to the sequence of classes 
$$ S = \{ D_{1}, D_{2}, D_{3}, \ell_{12}^{++}, \ell_{12}^{+-}, \ell_{13}^{++}, \ell_{23}^{++}, \ell_{12}^{-+}, \ell_{13}^{-+}, \ell_{23}^{--}, \overline{\ell_{12}^{++}}, \overline{\ell_{12}^{+-}}, \overline{\ell_{13}^{++}}, \overline{\ell_{23}^{++}}, \overline{\ell_{12}^{-+}}, \overline{\ell_{13}^{-+}}, C_{1}^{+-}, C_{2}^{+-} \} $$ is 
\[
\begin{pmatrix}
0 & 2 & 2 & 0 & 0 & 0 & 1 & 0 & 0 & 1 & 0 & 0 & 0 & 1 & 0 & 0 & 0 & 1\\
2 & 0 & 2 & 0 & 0 & 1 & 0 & 0 & 1 & 0 & 0 & 0 & 1 & 0 & 0 & 1 & 1 & 0\\
2 & 2 & 0 & 1 & 1 & 0 & 0 & 1 & 0 & 0 & 1 & 1 & 0 & 0 & 1 & 0 & 1 & 1\\
0 & 0 & 1 &-2 & 0 & 1 & 0 & 0 & 0 & 0 & 0 & 0 & 0 & 1 & 0 & 0 & 0 & 0\\
0 & 0 & 1 & 0 &-2 & 1 & 0 & 0 & 0 & 0 & 0 & 0 & 0 & 0 & 0 & 0 & 0 & 0\\
0 & 1 & 0 & 1 & 1 &-2 & 1 & 0 & 0 & 0 & 0 & 0 & 0 & 0 & 0 & 0 & 0 & 0\\
1 & 0 & 0 & 0 & 0 & 1 &-2 & 0 & 1 & 0 & 1 & 0 & 0 & 0 & 1 & 0 & 0 & 0\\
0 & 0 & 1 & 0 & 0 & 0 & 0 &-2 & 1 & 0 & 0 & 0 & 0 & 1 & 0 & 0 & 0 & 0\\
0 & 1 & 0 & 0 & 0 & 0 & 1 & 1 &-2 & 0 & 0 & 0 & 0 & 0 & 0 & 0 & 0 & 1\\
1 & 0 & 0 & 0 & 0 & 0 & 0 & 0 & 0 &-2 & 0 & 1 & 0 & 0 & 0 & 0 & 0 & 0\\
0 & 0 & 1 & 0 & 0 & 0 & 1 & 0 & 0 & 0 &-2 & 0 & 1 & 0 & 0 & 0 & 0 & 0\\
0 & 0 & 1 & 0 & 0 & 0 & 0 & 0 & 0 & 1 & 0 &-2 & 1 & 0 & 0 & 0 & 0 & 0\\
0 & 1 & 0 & 0 & 0 & 0 & 0 & 0 & 0 & 0 & 1 & 1 &-2 & 1 & 0 & 0 & 0 & 0\\
1 & 0 & 0 & 1 & 0 & 0 & 0 & 1 & 0 & 0 & 0 & 0 & 1 &-2 & 0 & 1 & 0 & 0\\
0 & 0 & 1 & 0 & 0 & 0 & 1 & 0 & 0 & 0 & 0 & 0 & 0 & 0 &-2 & 1 & 0 & 0\\
0 & 1 & 0 & 0 & 0 & 0 & 0 & 0 & 0 & 0 & 0 & 0 & 0 & 1 & 1 &-2 & 0 & 1\\
0 & 1 & 1 & 0 & 0 & 0 & 0 & 0 & 0 & 0 & 0 & 0 & 0 & 0 & 0 & 0 & -2& 1\\
1 & 0 & 1 & 0 & 0 & 0 & 0 & 0 & 1 & 0 & 0 & 0 & 0 & 0 & 0 & 1 & 1 & -2
\end{pmatrix},
\]
so it has determinant $-48$, which is nonzero. Consequently, the above classes are also linearly independent, so the Picard number of $\overline{W}$ is at least $18$.
\\~\\
\indent Under our assumption on $k$, one can check easily that $W_{5}$ is smooth, so $W$ has good reduction at $p = 5$. We will now show that the Picard number of $\overline{W}_{5}$ equals exactly $18$. Let $\overline{W}_{5}$ be the base change of $W_{5}$ to an algebraic closure of $\Fbb_{5}$, and $F : \overline{W}_{5} \rightarrow \overline{W}_{5}$ the geometric Frobenius morphism, defined by $([x:r],[y:s],[z:t]) \mapsto ([x^{5}:r^{5}], [y^{5}:s^{5}], [z^{5}:t^{5}])$. Choose a prime $l \not= 5$ and let $F^{*}$ be the endomorphism of $\textup{H}^{2}_{\textup{\'et}}(\overline{W}_{5}, \Qbb_{l}(1))$ induced by $F$. By Proposition 3.4, the Picard rank of $\overline{W}$ is bounded above by that of $\overline{W}_{5}$, which in turn is at most the number of eigenvalues of $F^{*}$ that are roots of unity. As in \cite{vL07a}, we find the characteristic polynomial of $F^{*}$ by counting points on $W_{5}$. Almost all fibers of the fibration $\pi_{1}$ are smooth curves of genus 1. Using \textsc{Magma} we can count the number of points over small fields fiber by fiber. The first three results are: 
$$ W_{5}(\Fbb_{5}) = 42, \hspace{0.5cm} W_{5}(\Fbb_{5^{2}}) = 1032, \hspace{0.5cm} W_{5}(\Fbb_{5^{3}}) = 16122. $$ 

From the Lefschetz fixed point formula, we find that the trace of the $n$-th power of Frobenius acting on $\textup{H}^{2}_{\textup{\'et}}(\overline{W}_{5}, \Qbb_{l})$ equals $\# W_{5}(\Fbb_{5^{n}}) - 5^{2n} - 1$; the trace on the Tate twist $\textup{H}^{2}_{\textup{\'et}}(\overline{W}_{5}, \Qbb_{l}(1))$ is obtained by \emph{dividing} by $5^{n}$. Meanwhile, on the subspace $V \subset \textup{H}^{2}_{\textup{\'et}}(\overline{W}_{5}, \Qbb_{l}(1))$ generated by the above $18$ divisor classes, as the characteristic polynomial of the Frobenius acting on $V$ is $(t-1)^{11} (t+1)^{7}$, the trace $t_{n}$ is equal to $18$ if $n$ is even, and equal to $4$ if $n$ is odd. Hence, on the $4$-dimensional quotient $Q = \textup{H}^{2}_{\textup{\'et}}(\overline{W}_{5}, \Qbb_{l}(1))/V$ , the trace equals 
$$ \frac{\# W_{5}(\Fbb_{5^{n}})}{5^{n}} - 5^{n} - \frac{1}{5^{n}} - t_{n}. $$ 
These traces are sums of powers of eigenvalues, and we use the Newton identities to compute the elementary symmetric polynomials in these eigenvalues, which are the coefficients of the characteristic polynomial $f$ of the Frobenius acting on $Q$ (see \cite{vL07b}, Lemma 2.4). This yields the first half of the coefficients of $f$, including the middle coefficient, which turns out to be non-zero. This implies that the sign in the functional equation $t^{4}f(1/t) = \pm f(t)$ is $+1$, so this functional equation determines $f$, which we calculate to be
$$ f(t) = t^{4} + \frac{4}{5}t^{3} + \frac{6}{5}t^{2} + \frac{4}{5}t + 1. $$
As a result, we find that the characteristic polynomial of Frobenius acting on $\textup{H}^{2}_{\textup{\'et}}(\overline{W}_{5}, \Qbb_{l}(1))$ is equal to $(t-1)^{11} (t+1)^{7} f$. The polynomial $5f \in \Zbb[t]$ is irreducible, primitive and not monic, so its roots are not roots of unity. Thus we obtain an upper bound of $18$ for the Picard number of $\overline{W}$.

Therefore, we deduce that $\textup{rk}\,\textup{Pic}\,\overline{W} = 18$, and the sequence $S$ of $18$ divisor classes form a sublattice $\Lambda \subset \textup{NS}\,\overline{W} = \textup{Pic}\,\overline{W}$. We now verify that this is actually the whole lattice. Indeed, assume that $\Lambda$ is a proper sublattice of $\textup{NS}\,\overline{W}$, hence their discriminants differ by a square factor. We know that $\textup{disc}(\textup{NS}\,\overline{W}) = -48 = -3.2^{4}$, so $\Lambda$ has to be a sublattice of index $2$ or $4$. In other words, there would exist a divisor class of the form 
$$ E = \frac{1}{2} \sum_{E_{i} \in S} a_{i}E_{i}, \hspace{0.5cm} a_{i} \in \{0,1\} $$
in $\textup{Pic}\,\overline{W}$. With the condition that all the intersection pairings between $E$ and every divisor class in $S$ give \emph{integer} values, we find that there are only two possibilities:
\begin{enumerate}
	\item[(a)] $E = \frac{1}{2}(D_{1} + \ell_{12}^{+-} + \ell_{23}^{++} + \ell_{12}^{-+} + \ell_{23}^{--} + \overline{\ell_{12}^{++}} + \overline{\ell_{12}^{+-}} + \overline{\ell_{13}^{++}} + \overline{\ell_{13}^{-+}})$;
	
	\item[(b)] $E = \frac{1}{2}(D_{2} + D_{3} + \ell_{13}^{++} + \ell_{23}^{++} + \ell_{13}^{-+} + \ell_{23}^{--} + \overline{\ell_{12}^{++}} + \overline{\ell_{12}^{-+}})$.
\end{enumerate}
In the first case, we can check that $E^{2} = -1$ is odd, which is a contradiction since the intersection pairing on $\textup{Pic}\,\overline{W}$ is \emph{even}. In the second case, we have $E^{2} = 2$, which is even. However, using the fact that in $\textup{Pic}\,\overline{W}$:
$$ D_{3} = \ell_{13}^{++} + \ell_{13}^{-+} + \ell_{23}^{++} + \ell_{23}^{-+} $$
and
$$ D_{2} = \ell_{21}^{-+} + \ell_{21}^{--} + \ell_{23}^{-+} + \ell_{23}^{--} = \overline{\ell_{12}^{+-}} + \overline{\ell_{12}^{--}} + \ell_{23}^{-+} + \ell_{23}^{--}, $$
we can rewrite 
$$ E = D_{3} + \ell_{23}^{--} + \frac{1}{2}(\overline{\ell_{12}^{++}} + \overline{\ell_{12}^{+-}} + \overline{\ell_{12}^{-+}} + \overline{\ell_{12}^{--}}). $$
This implies that if (b) were true, then we would have $\frac{1}{2}(\overline{\ell_{12}^{++}} + \overline{\ell_{12}^{+-}} + \overline{\ell_{12}^{-+}} + \overline{\ell_{12}^{--}}) \in \textup{Pic}\,\overline{W}$. By contrast, using the argument in the proof of \cite{Nik75}, Lemma 3, one shows that the sum of divisor classes of \emph{four} non-singular, non-intersecting rational curves on a K3 surface cannot be divisible by $2$, since the total number of elements in such a set of classes can only be $0$, $8$, or $16$. This is a contradiction, so the lattice generated by $S$ is indeed the whole N\'eron--Severi lattice of $\overline{W}$, thus completing our proof.
\end{proof}

\begin{remark}
The above 18 divisor classes that form a basis of $\textup{Pic}\,\overline{W}$ are not unique, because one can find other first $16$ divisors in the set of $D_{i}$ and $\ell_{ij}^{\pm \pm}, \overline{\ell_{ij}^{\pm \pm}}$ for $1 \leq i \not= j \leq 3$, and find the other $2$ remaining divisors in the set of $C_{i}^{\pm \pm}$ for $1 \leq i \leq 3$ with different indexes $i$. Note that the divisors $A_{1}, A_{2}, A_{3}$ and $B_{1}, B_{2}, B_{3}$ defined by \emph{irreducible} singular fibers have the same classes as $D_{1}, D_{2}, D_{3}$, respectively.
\end{remark}

Next, we consider the geometric Picard group of the affine surface $U$ defined by the same equation.

\begin{corollary}
Let $U \subset W$ be the affine surface defined by the same equation 
$$ x^{2}+y^{2}+z^{2} + 4(x^{2}y^{2}+y^{2}z^{2}+z^{2}x^{2}) - 16x^{2}y^{2}z^{2} = k, $$
where $k \in \Zbb$. Then the Picard number of $\overline{U} = U_{\overline{\Qbb}}$ equals $15$.
\end{corollary}

\begin{proof}
By the exact sequence in Proposition 3.1, we obtain
$$ \textup{Pic}\,\overline{U} \cong \textup{Pic}\,\overline{W}/(\Zbb D_{1} \oplus \Zbb D_{2} \oplus \Zbb D_{3}), $$
so $\textup{Pic}\,\overline{U}$ is \emph{free} and the Picard number of $\overline{U}$ is equal to $18 - 3 = 15$.
\end{proof}

\subsection{The algebraic Brauer group}
Now given the geometric Picard group, we can compute directly the algebraic Brauer group of the Markoff-type cubic surfaces in question.

\begin{theorem}
For $k \in \Zbb$, let $W \subset \Pbb^{1} \times \Pbb^{1} \times \Pbb^{1}$ be the MK3 surface defined over $\Qbb$ by the $(2,2,2)$-form 
\begin{equation}
    F(x,y,z) = x^{2} + y^{2} + z^{2} + 4(x^{2}y^{2} + y^{2}z^{2} + z^{2}x^{2}) - 16x^{2}y^{2}z^{2} - k = 0.
\end{equation}
Consider the field extension $K := \Qbb(\sqrt{-1}, \sqrt{\alpha},\sqrt{\bar{\alpha}})$ where $\alpha, \bar{\alpha}$ are given as in the proof of Proposition 3.5. If $k$ satisfies the following conditions:
\begin{enumerate}
    \item None of $2(4k+1), \Delta, 2(4k+1)\Delta$ is a square in $\Qbb$;
    \item $k \equiv 3$ \textup{mod} $5$,
\end{enumerate}
such that $G := \textup{Gal}(K/\Qbb) \cong D_{4} \times \Zbb/2\Zbb$ (this is the most general case of the field extension over which all the divisor classes are defined), then
$$ \textup{Br}_{1}\,W/\textup{Br}_{0}\,W \cong (\Zbb/2\Zbb)^{3}. $$
Furthermore, for the affine subscheme $U = W \setminus \{rst = 0\}$, we even have 
$$ \textup{Br}_{1}\,U/\textup{Br}_{0}\,U \cong (\Zbb/2\Zbb)^{4}. $$
\end{theorem}

\begin{proof}
Since $W$ is smooth, projective, geometrically integral over $K$, we have $\overline{\Qbb}[W]^{\times} = \overline{\Qbb}^{\times}$. One already has $W(\Qbb) \not= \emptyset$. By the Hochschild--Serre spectral sequence, we have an isomorphism 
$$ \textup{Br}_{1}\,W/\textup{Br}_{0}\,W \simeq \textup{H}^{1}(\Qbb,\textup{Pic}\,\overline{W}). $$

By Proposition 3.5, the geometric Picard number of $W$ is equal to $18$ and a basis of $\textup{Pic}\,\overline{W}$ is given by 
$$ S = \{ D_{1}, D_{2}, D_{3}, \ell_{12}^{++}, \ell_{12}^{+-}, \ell_{13}^{++}, \ell_{23}^{++}, \ell_{12}^{-+}, \ell_{13}^{-+}, \ell_{23}^{--}, \overline{\ell_{12}^{++}}, \overline{\ell_{12}^{+-}}, \overline{\ell_{13}^{++}}, \overline{\ell_{23}^{++}}, \overline{\ell_{12}^{-+}}, \overline{\ell_{13}^{-+}}, C_{1}^{+-}, C_{2}^{+-} \} $$
along with the intersection matrix. If we denote by $(S)$ the column vector of elements of $S$, then from the intersection pairings of the classes in $S$ with the other classes in the list of Proposition 3.5, we find that
\[ 
\begin{pmatrix}
\overline{\ell_{12}^{--}} \\ \ell_{12}^{--} \\ \overline{\ell_{23}^{--}} \\ \ell_{13}^{--} \\ \overline{\ell_{13}^{--}} \\ \ell_{13}^{+-} \\ \overline{\ell_{13}^{+-}} \\ \ell_{23}^{+-} \\ \overline{\ell_{23}^{+-}} \\ \ell_{23}^{-+} \\ \overline{\ell_{23}^{-+}}
\end{pmatrix}
= 
\begin{pmatrix}
0& 1& -1& 0& 0& 1& 1& 0& 1& -1& 0& -1& 0& 0& 0& 0& 0& 0\\
2& 1& -1& -1& -1& -1& -1& -1& -1& 1& -1& 0& 0& 0& -1& 0& 0& 0\\
-2& 0& 0& 1& 0& 1& 1& 1& 1& -1& 1& 0& 1& 1& 1& 1& 0& 0\\
-1& -1& 1& 1& 1& 1& 1& 0& 0& -1& 1& 0& 0& 0& 1& 0& 0& 0\\
1& -1& 1& 0& 0& -1& -1& 0& -1& 1& 0& 1& 0& 0& -1& -1& 0& 0\\
1& 0& 0& -1& -1& -1& 0& 0& 0& 0& 0& 0& 0& 0& 0& 0& 0& 0\\
1& 0& 0& 0& 0& 0& 0& 0& 0& 0& -1& -1& -1& 0& 0& 0& 0& 0\\
0& 1& 0& 0& 0& 0& -1& 0& 0& 0& -1& 0& 0& 0& -1& 0& 0& 0\\
0& 1& 0& -1& 0& 0& 0& -1& 0& 0& 0& 0& 0& -1& 0& 0& 0& 0\\
0& 0& 1& 0& 0& -1& -1& 0& -1& 0& 0& 0& 0& 0& 0& 0& 0& 0\\
0& 0& 1& 0& 0& 0& 0& 0& 0& 0& 0& 0& -1& -1& 0& -1& 0& 0
\end{pmatrix}
(S)
\]
and also
\begin{equation}
\begin{cases}
D_{1} = \sum_{\epsilon = \pm\,\textup{fixed}, \delta = \pm\,\textup{varied}} \sum_{j \in \{2,3\}} \ell_{1j}^{\epsilon\delta} = \sum_{\epsilon = \pm\,\textup{fixed}, \delta = \pm\,\textup{varied}} \sum_{j \in \{2,3\}} \overline{\ell_{1j}^{\epsilon\delta}}, \\
D_{2} = \sum_{\epsilon = \pm\,\textup{fixed}, \delta = \pm\,\textup{varied}} \sum_{j \in \{1,3\}} \ell_{1j}^{\epsilon\delta} = \sum_{\epsilon = \pm\,\textup{fixed}, \delta = \pm\,\textup{varied}} \sum_{j \in \{1,3\}} \overline{\ell_{1j}^{\epsilon\delta}}, \\
D_{3} = \sum_{\epsilon = \pm\,\textup{fixed}, \delta = \pm\,\textup{varied}} \sum_{j \in \{1,2\}} \ell_{1j}^{\epsilon\delta} = \sum_{\epsilon = \pm\,\textup{fixed}, \delta = \pm\,\textup{varied}} \sum_{j \in \{1,2\}} \overline{\ell_{1j}^{\epsilon\delta}}, \\
C_{1}^{--} = C_{1}^{+-} ; C_{1}^{++} = D_{1} - C_{1}^{+-}, \\
C_{2}^{--} = C_{2}^{+-} ; C_{2}^{++} = D_{2} - C_{2}^{+-}, \\
C_{3}^{--} = C_{3}^{+-} = \ell_{12}^{++} + \ell_{13}^{++} + \ell_{23}^{++} + \overline{\ell_{12}^{++}} + \overline{\ell_{13}^{++}} + \overline{\ell_{23}^{++}} - C_{1}^{+-} - C_{2}^{+-} ; C_{3}^{++} = D_{3} - C_{3}^{+-}.
\end{cases}
\end{equation}

Now we study the action of the absolute Galois group on $\textup{Pic}\,\overline{W}$, which can be reduced to the action of $G = \textup{Gal}(K/\Qbb)$. One clearly has $G \cong D_{4} \times \Zbb/2\Zbb \cong (\langle \sigma \rangle \rtimes \langle \tau \rangle) \times \langle \rho \rangle$, where
$$ \sigma(\alpha) = \bar{\alpha}, \hspace{0.25cm} \sigma(\bar{\alpha}) = -\alpha, $$
$$ \tau(\alpha) = \alpha, \hspace{0.25cm} \tau(\bar{\alpha}) = -\bar{\alpha}, $$
$$ \rho(\sqrt{-1}) = -\sqrt{-1}. $$
Note that for $1 \leq i \not= j \leq 3$, $\sigma(\ell_{ij}^{\pm \pm}) = \overline{\ell_{ij}^{\pm \mp}}$, $\sigma(\overline{\ell_{ij}^{\pm \pm}}) = \ell_{ij}^{\mp \pm}$; $\tau(\ell_{ij}^{\pm \pm}) = \ell_{ij}^{\pm \mp}$, $\tau(\overline{\ell_{ij}^{\pm \pm}}) = \overline{\ell_{ij}^{\mp \pm}}$; $\rho(C_{i}^{\pm \pm}) = C_{i}^{\mp \pm}$ and $\sigma(C_{i}^{\pm \pm}) = C_{i}^{\pm \mp} = D_{i} - C_{i}^{\pm \pm}$. We have the following matrix of $\langle\sigma\rangle$ acting \emph{stably} on the first $16$ divisor classes of $\textup{Pic}\,\overline{W}$ in the ordered basis given by $(S)$ (since $\sigma$ acts trivially on $C_{i}^{\pm \pm}$ for all $1 \leq i \leq 3$):
\[
\begin{pmatrix}
1 & 0 & 0 & 0 & 0 & 0 & 0 & 0 & 0 & 0 & 0 & 0 & 0 & 0 & 0 & 0\\ 
0 & 1 & 0 & 0 & 0 & 0 & 0 & 0 & 0 & 0 & 0 & 0 & 0 & 0 & 0 & 0\\ 
0 & 0 & 1 & 0 & 0 & 0 & 0 & 0 & 0 & 0 & 0 & 0 & 0 & 0 & 0 & 0\\ 
0 & 0 & 0 & 0 & 0 & 0 & 0 & 1 & 0 & 0 & 0 & 0 & 0 & 0 & 0 & 0\\ 
0 & 0 & 0 & 0 & 0 & 0 & 1 & 0 & 0 & 0 & 0 & 0 & 0 & 0 & 0 & 0\\ 
1 & 0 & 0 & 0 & 0 & 0 &-1 &-1 & 0 & 0 & 0 &-1 & 0 & 0 & 0 & 0\\ 
0 & 1 & 0 &-1 & 0 &-1 & 0 & 0 & 0 & 0 & 0 & 0 & 0 & 0 & 0 &-1\\ 
0 & 1 &-1 & 0 & 0 & 0 & 0 &-1 & 0 & 1 & 1 & 0 & 0 & 1 &-1 & 0\\ 
1 &-1 & 1 & 0 & 0 & 0 & 0 & 1 &-1 &-1 &-1 & 0 &-1 &-1 & 1 & 0\\ 
0 & 0 & 1 & 0 & 0 & 0 & 0 & 0 & 0 & 0 & 0 &-1 &-1 & 0 & 0 &-1\\ 
0 & 0 & 0 & 0 & 0 & 1 & 0 & 0 & 0 & 0 & 0 & 0 & 0 & 0 & 0 & 0\\ 
2 & 1 &-1 &-1 &-1 &-1 &-1 & 0 &-1 &-1 &-1 & 0 & 0 &-1 & 1 & 0\\ 
0 & 0 & 0 & 0 & 0 & 0 & 0 & 0 & 0 & 0 & 1 & 0 & 0 & 0 & 0 & 0\\ 
0 & 0 & 1 & 0 & 0 & 0 & 0 & 0 & 0 &-1 &-1 & 0 & 0 &-1 & 0 & 0\\ 
0 & 0 & 0 & 1 & 0 & 0 & 0 & 0 & 0 & 0 & 0 & 0 & 0 & 0 & 0 & 0\\ 
0 & 0 & 0 & 0 & 0 & 0 & 0 & 0 & 0 & 1 & 0 & 0 & 0 & 0 & 0 & 0 
\end{pmatrix}.
\]
Hence, we obtain
$$ \textup{Ker}(1+\rho) = \langle C_{1}^{+-} - C_{1}^{--}, C_{2}^{+-} - C_{2}^{--} \rangle, $$
and $\textup{Ker}(1+\sigma+\sigma^{2}+\sigma^{3}) = \langle D_{1}-\overline{\ell_{12}^{-+}}-\overline{\ell_{12}^{++}}-\overline{\ell_{13}^{++}}-\overline{\ell_{13}^{-+}}, D_{2}-\overline{\ell_{12}^{-+}}-\overline{\ell_{12}^{++}}-\ell_{23}^{--}-\overline{\ell_{23}^{++}}, D_{3}-\overline{\ell_{13}^{++}}-\overline{\ell_{13}^{-+}}-\ell_{23}^{--}-\overline{\ell_{23}^{++}}, \ell_{12}^{++} - \overline{\ell_{12}^{-+}}, \ell_{12}^{+-}-\overline{\ell_{12}^{++}}, \ell_{12}^{-+}-\overline{\ell_{12}^{++}}, \overline{\ell_{12}^{+-}}-\overline{\ell_{12}^{-+}}, \ell_{13}^{++}-\overline{\ell_{13}^{-+}}, \ell_{13}^{-+}-\overline{\ell_{13}^{++}}, \ell_{23}^{++}-\ell_{23}^{--} \rangle$. We also have
$$ \textup{Ker}(1-\rho) = \langle D_{1}, D_{2}, D_{3}, \ell_{12}^{++}, \ell_{12}^{+-}, \ell_{13}^{++}, \ell_{23}^{++}, \ell_{12}^{-+}, \ell_{13}^{-+}, \ell_{23}^{--}, \overline{\ell_{12}^{++}}, \overline{\ell_{12}^{+-}}, \overline{\ell_{13}^{++}}, \overline{\ell_{23}^{++}}, \overline{\ell_{12}^{-+}}, \overline{\ell_{13}^{-+}} \rangle, $$
and 
$\textup{Ker}(1-\sigma) \cap \textup{Pic}\,\overline{W}^{\langle \rho \rangle} = \langle D_{1}, D_{2}, D_{3}, \ell_{12}^{+-}+\ell_{12}^{++}-\overline{\ell_{12}^{+-}}+\overline{\ell_{12}^{-+}}+2\ell_{13}^{++}-\overline{\ell_{13}^{++}}+\overline{\ell_{13}^{-+}}+\ell_{23}^{++}-\ell_{23}^{--}, \ell_{12}^{-+}+\ell_{12}^{++}+\overline{\ell_{13}^{++}}+\overline{\ell_{13}^{-+}}-\ell_{23}^{++}-\ell_{23}^{--}+2\overline{\ell_{23}^{++}}, \overline{\ell_{12}^{++}}-2\ell_{12}^{++}-\overline{\ell_{12}^{-+}}-\ell_{13}^{++}+\ell_{13}^{-+}-2\overline{\ell_{13}^{-+}}+\ell_{23}^{++}+\ell_{23}^{--}-2\overline{\ell_{23}^{++}} \rangle$.
\\~\\
\indent Given a finite cyclic group $G = \langle \sigma \rangle$ and a $G$-module $M$, by \cite{NSW15}, Proposition 1.7.1, recall that we have isomorphisms $\textup{H}^{1}(G,M) \cong \hat{\textup{H}}^{-1}(G,M)$, where the latter group is the quotient of $\prescript{}{N_{G}}{M}$, the set of elements of $M$ of norm $0$, by its subgroup $(1 - \sigma)M$.

By \cite{NSW15}, Proposition 1.6.7, we have 
$$ \textup{H}^{1}(\Qbb,\textup{Pic}\,\overline{W}) = \textup{H}^{1}(G,\textup{Pic}\,\overline{W}), $$ where $G = (\langle \sigma \rangle \rtimes \langle \tau \rangle) \times \langle \rho \rangle \cong D_{4} \times \Zbb/2\Zbb$. Then one has the following (inflation-restriction) exact sequence
$$ 0 \rightarrow \textup{H}^{1}((\langle \tau \rangle \ltimes \langle \sigma \rangle), \textup{Pic}\,\overline{W}^{\langle \rho \rangle}) \rightarrow \textup{H}^{1}(G, \textup{Pic}\,\overline{W}) \rightarrow \textup{H}^{1}(\langle \rho \rangle, \textup{Pic}\,\overline{W}) = \frac{\textup{Ker}(1+\rho)}{(1-\rho)\textup{Pic}\,\overline{W}} = 0, $$
so $\textup{H}^{1}(G, \textup{Pic}\,\overline{W}) \cong \textup{H}^{1}((\langle \tau \rangle \ltimes \langle \sigma \rangle), \textup{Pic}\,\overline{W}^{\langle \rho \rangle})$. Now we are left with 
\[
\begin{aligned}
0 \rightarrow \textup{H}^{1}(\langle \tau \rangle, \textup{Pic}\,\overline{W}^{\langle \sigma, \rho \rangle}) &\rightarrow \textup{H}^{1}(G, \textup{Pic}\,\overline{W})\\ &\rightarrow \textup{H}^{1}(\langle \sigma \rangle, \textup{Pic}\,\overline{W}^{\langle \rho \rangle}) = \frac{\textup{Ker}(1+\sigma+\sigma^{2}+\sigma^{3}) \cap \textup{Pic}\,\overline{W}^{\langle \rho \rangle}}{(1-\sigma)\textup{Pic}\,\overline{W}^{\langle \rho_{2} \rangle}} = 0,
\end{aligned}
\]
so $\textup{H}^{1}(G, \textup{Pic}\,\overline{W}) \cong \textup{H}^{1}(\langle \tau \rangle, \textup{Pic}\,\overline{W}^{\langle \sigma, \rho \rangle})$. The latter group can be computed as follows. We already have 
$$ \textup{Pic}\,\overline{W}^{\langle \sigma, \rho \rangle} = \textup{Ker}(1-\sigma) \cap \textup{Pic}\,\overline{W}^{\langle \rho \rangle}. $$
We find that 
$ \textup{H}^{1}(\Qbb,\textup{Pic}\,\overline{W}) = \textup{H}^{1}(G,\textup{Pic}\,\overline{W})\\ \cong [\textup{Ker}(1+\tau) \cap \textup{Pic}\,\overline{W}^{\langle \sigma, \rho \rangle}] / (1-\tau)\textup{Pic}\,\overline{W}^{\langle \sigma, \rho \rangle}\\ = \langle \ell_{12}^{+-}+\ell_{12}^{++}-\overline{\ell_{12}^{+-}}+\overline{\ell_{12}^{-+}}+2\ell_{13}^{++}-\overline{\ell_{13}^{++}}+\overline{\ell_{13}^{-+}}+\ell_{23}^{++}-\ell_{23}^{--} - D_{1}, \ell_{12}^{-+}+\ell_{12}^{++}+\overline{\ell_{13}^{++}}+\overline{\ell_{13}^{-+}}-\ell_{23}^{++}-\ell_{23}^{--}+2\overline{\ell_{23}^{++}} - D_{1}, \overline{\ell_{12}^{++}}-2\ell_{12}^{++}-\overline{\ell_{12}^{-+}}-\ell_{13}^{++}+\ell_{13}^{-+}-2\overline{\ell_{13}^{-+}}+\ell_{23}^{++}+\ell_{23}^{--}-2\overline{\ell_{23}^{++}} + D_{1} \rangle\\ / 2\langle \ell_{12}^{+-}+\ell_{12}^{++}-\overline{\ell_{12}^{+-}}+\overline{\ell_{12}^{-+}}+2\ell_{13}^{++}-\overline{\ell_{13}^{++}}+\overline{\ell_{13}^{-+}}+\ell_{23}^{++}-\ell_{23}^{--} - D_{1}, \ell_{12}^{-+}+\ell_{12}^{++}+\overline{\ell_{13}^{++}}+\overline{\ell_{13}^{-+}}-\ell_{23}^{++}-\ell_{23}^{--}+2\overline{\ell_{23}^{++}} - D_{1}, \overline{\ell_{12}^{++}}-2\ell_{12}^{++}-\overline{\ell_{12}^{-+}}-\ell_{13}^{++}+\ell_{13}^{-+}-2\overline{\ell_{13}^{-+}}+\ell_{23}^{++}+\ell_{23}^{--}-2\overline{\ell_{23}^{++}} + D_{1} \rangle\\ \cong (\Zbb/2\Zbb)^{3}$.
\\~\\
\indent We keep the notation as above. Now $\textup{Pic}\,\overline{U}$ is given by the following quotient group
$$ \textup{Pic}\,\overline{U} \cong \textup{Pic}\,\overline{W}/(\Zbb D_{1} \oplus \Zbb D_{2} \oplus \Zbb D_{3}) $$
by Proposition 3.1. Here for any divisor $D \in \textup{Pic}\,\overline{X}$, denote by $[D]$ its image in $\textup{Pic}\,\overline{U}$. By Proposition 3.1, we also have $\overline{\Qbb}^{\times} = \overline{\Qbb}[U]^{\times}$. By the Hochschild--Serre spectral sequence, we have the following injective homomorphism 
$$ \textup{Br}_{1}\,U/\textup{Br}_{0}\,U \cong \textup{H}^{1}(\Qbb,\textup{Pic}\,\overline{U}) $$
as $K$ is a number field. Since $\textup{Pic}\,\overline{U}$ is free and $\textup{Gal}(\overline{\Qbb}/K)$ acts on $\textup{Pic}\,\overline{U}$ trivially, we obtain that $\textup{H}^{1}(\Qbb,\textup{Pic}\,\overline{U}) \cong \textup{H}^{1}(G,\textup{Pic}\,\overline{U})$. With the action of $G$, we can compute in the quotient group $\textup{Pic}\,\overline{U}$:
\begin{equation*}
\begin{cases}
\textup{Ker}(1+\rho) = \langle [C_{1}^{+-}]-[C_{1}^{--}], [C_{2}^{+-}]-[C_{2}^{--}] \rangle,\\

\textup{Ker}(1-\rho) = \langle [\ell_{12}^{++}], [\ell_{12}^{+-}], [\ell_{13}^{++}], [\ell_{23}^{++}], [\ell_{12}^{-+}], [\ell_{13}^{-+}], [\ell_{23}^{--}], [\overline{\ell_{12}^{++}}], [\overline{\ell_{12}^{+-}}], [\overline{\ell_{13}^{++}}], [\overline{\ell_{23}^{++}}], [\overline{\ell_{12}^{-+}}], [\overline{\ell_{13}^{-+}}] \rangle, 
\end{cases}
\end{equation*}
$\textup{Ker}(1-\sigma) \cap \textup{Pic}\,\overline{U}^{\langle \rho \rangle} = \langle [\ell_{12}^{+-}]+[\ell_{12}^{++}]-[\overline{\ell_{12}^{+-}}]+[\overline{\ell_{12}^{-+}}]+2[\ell_{13}^{++}]-[\overline{\ell_{13}^{++}}]+[\overline{\ell_{13}^{-+}}]+[\ell_{23}^{++}]-[\ell_{23}^{--}], [\ell_{12}^{-+}]+[\ell_{12}^{++}]+[\overline{\ell_{13}^{++}}]+[\overline{\ell_{13}^{-+}}]-[\ell_{23}^{++}]-[\ell_{23}^{--}]+2[\overline{\ell_{23}^{++}}], [\overline{\ell_{12}^{++}}]-2[\ell_{12}^{++}]-[\overline{\ell_{12}^{-+}}]-[\ell_{13}^{++}]+[\ell_{13}^{-+}]-2[\overline{\ell_{13}^{-+}}]+[\ell_{23}^{++}]+[\ell_{23}^{--}]-2[\overline{\ell_{23}^{++}}] \rangle$,\\
and\\
$\textup{Ker}(1+\sigma+\sigma^{2}+\sigma^{3}) = \langle [\ell_{12}^{++}] - [\overline{\ell_{12}^{-+}}], [\ell_{12}^{+-}] + [\overline{\ell_{12}^{-+}}], [\ell_{12}^{-+}] + [\overline{\ell_{12}^{-+}}], [\overline{\ell_{12}^{++}}] + [\overline{\ell_{12}^{-+}}], [\overline{\ell_{12}^{+-}}] - [\overline{\ell_{12}^{-+}}], [\ell_{13}^{++}] - [\overline{\ell_{13}^{-+}}], [\ell_{13}^{-+}] + [\overline{\ell_{13}^{-+}}], [\overline{\ell_{13}^{++}}] + [\overline{\ell_{13}^{-+}}], [\ell_{23}^{++}] + [\overline{\ell_{23}^{++}}], [\ell_{23}^{--}] + [\overline{\ell_{23}^{++}}] \rangle$. Then 
$$ \textup{H}^{1}(\langle \rho \rangle, \textup{Pic}\,\overline{U}) = \frac{\textup{Ker}(1+\rho)}{(1-\rho)\textup{Pic}\,\overline{U}} = 0, $$
\[
\textup{H}^{1}(\langle \sigma \rangle, \textup{Pic}\,\overline{U}^{\langle \rho \rangle}) = \frac{\textup{Ker}(1+\sigma+\sigma^{2}+\sigma^{3}) \cap \textup{Pic}\,\overline{U}^{\langle \rho \rangle}}{(1-\sigma)\textup{Pic}\,\overline{U}^{\langle \rho \rangle}} = \frac{\langle [\ell_{12}^{++}]+[\overline{\ell_{12}^{++}}] \rangle}{2\langle [\ell_{12}^{++}]+[\overline{\ell_{12}^{++}}] \rangle} \cong \Zbb/2\Zbb,
\]
and
$$ \textup{H}^{1}(\langle \tau \rangle, \textup{Pic}\,\overline{U}^{\langle \sigma, \rho \rangle}) \cong \frac{\textup{Ker}(1+\tau) \cap \textup{Pic}\,\overline{U}^{\langle \sigma, \rho \rangle}}{(1-\tau)\textup{Pic}\,\overline{U}^{\langle \sigma, \rho \rangle}} = \frac{\textup{Pic}\,\overline{U}^{\langle \sigma, \rho \rangle}}{2\textup{Pic}\,\overline{U}^{\langle \sigma, \rho \rangle}} \cong (\Zbb/2\Zbb)^{3}. $$
Since $\textup{Pic}\,\overline{U}^{G} = 0$, from all the similar inflation-restriction exact sequences, we conclude that
$$ \textup{Br}_{1}\,U/\textup{Br}_{0}\,U \cong \textup{H}^{1}(G,\textup{Pic}\,\overline{U}) \cong (\Zbb/2\Zbb)^{4}. $$

Now we produce some concrete generators in $\textup{Br}_{1}\,U$ for $\textup{Br}_{1}\,U/\textup{Br}_{0}\,U$. The affine scheme $U \subset \Abb^{3}$ is defined over $\Qbb$ by the equation
\begin{equation}
x^{2} + y^{2} + z^{2} + 4(x^{2}y^{2} + y^{2}z^{2} + z^{2}x^{2}) - 16x^{2}y^{2}z^{2} = k.
\end{equation}
This affine equation is equivalent to
\begin{equation}
(4x^{2}+1)(4y^{2}+1)(4z^{2}+1) = (4k+1) + 128x^{2}y^{2}z^{2},
\end{equation}
\begin{equation}
(4x^{2}+1)(1+4y^{2}+4z^{2}-16y^{2}z^{2}) = (4k+1) - 32y^{2}z^{2},
\end{equation}
and also implies the following equation over $\{xyz \not= 0\}$:
\begin{equation}
(16x^{2}y^{2}-4x^{2}-4y^{2}-1)(16x^{2}z^{2}-4x^{2}-4z^{2}-1) = 2\left(\left(4x^{2} - \frac{4k-1}{4}\right)^{2} - \frac{(4k-5)^{2}-32}{16}\right),
\end{equation}
as well as similar ones obtained by permutation of coordinates in all the above equations. By Grothendieck's purity theorem, for any smooth variety $Y$ over a field $k$ of characteristic $0$, we have the exact sequence
$$ 0 \rightarrow \textup{Br}\,Y \rightarrow \textup{Br}\,k(Y) \rightarrow \oplus_{D \in Y^{(1)}} \textup{H}^{1}(k(D),\Qbb/\Zbb), $$
where the last map is given by the residue along the codimension-one point $D$. Hence, to prove that our quaternion algebras $\mathcal{A}_{1},\mathcal{A}_{2},\mathcal{B}$ come from a class in $\textup{Br}\,U$, it suffices to show that all their residues along the $15$ divisor classes $[\ell_{12}^{++}],\dots,[\ell_{13}^{-+}], [C_{1}^{+-}], [C_{2}^{+-}]$ generating $\textup{Pic}\,\overline{U}$ are trivial. However, in the function field of any such irreducible divisor, $-2(4k+1)$ or $(4k-5)^{2}-32$ is clearly a square; standard formulae for residues in terms of the tame symbol \cite{GS17}, Example 7.1.5, Proposition 7.5.1, therefore show that $\mathcal{A}_{1},\mathcal{A}_{2},\mathcal{B}$ are unramified, hence they are elements of $\textup{Br}\,U$ and moreover they are clearly algebraic. Since 
$$ \{4x^{2}+1 = 0\} \cap \{(4y^{2}+1)(4z^{2}+1)=0\}, $$
$$ \{4x^{2}+1 = 0\} \cap \{16y^{2}z^{2}-4y^{2}-4z^{2}-1=0\}, $$
$$ \{16x^{2}y^{2}-4x^{2}-4y^{2}-1=0\} \cap \{16x^{2}z^{2}-4x^{2}-4z^{2}-1=0\} $$
are closed subsets of codimension $2$ on $U$, we obtain that $$(4x^{2}+1,-2(4k+1)) = (2(4y^{2}+1)(4z^{2}+1),-2(4k+1)) = (2(16y^{2}z^{2}-4y^{2}-4z^{2}-1), 2(4k+1))$$ and $$(16x^{2}y^{2}-4x^{2}-4y^{2}-1, (4k-5)^{2}-32) = (2(16x^{2}z^{2}-4x^{2}-4z^{2}-1), (4k-5)^{2}-32)$$ in $\textup{Br}_{1}\,U$, as well as similar ones given by permutation of coordinates. Their residues at the irreducible divisors $D_{1}, D_{2}, D_{3}$ which form the complement of $U$ in $W$ are easily seen to be trivial. One thus also has $\mathcal{A}_{1},\mathcal{A}_{2},\mathcal{B} \in \textup{Br}_{1}\,W$. These elements, at least $\mathcal{A}_{1},\mathcal{A}_{2}$, are nontrivial, as they will contribute to the Brauer--Manin obstruction to the integral Hasse principle later.

In conclusion, we have $\textup{Br}_{1}\,W/\textup{Br}_{0}\,W \cong (\Zbb/2\Zbb)^{3}$, which can be viewed as a subgroup of $\textup{Br}_{1}\,U/\textup{Br}_{0}\,U \cong (\Zbb/2\Zbb)^{4}$.
\end{proof}

\begin{remark}
One can hope to find more explicit generators for the quotient group $\textup{Br}_{1}\,U/\textup{Br}_{0}\,U$ by studying further the equation and its geometric nature. Furthermore, it would be more interesting if one can compute the transcendental part of the Brauer group for this family of Markoff-type K3 surfaces like what the authors did for Markoff surfaces in \cite{LM20} and \cite{CTWX20}, which in general should be difficult. 
\end{remark}

\section{The Brauer--Manin obstruction}

\subsection{Review of the Brauer--Manin obstruction}
Here we briefly recall the Brauer--Manin obstruction in our setting, following \cite{Poo17}, Section 8.2 and \cite{CTX09}, Section 1. For each place $p$ of $\Qbb$ there is a pairing
$$ U(\Qpbb) \times \textup{Br}\,U \rightarrow \Qbb/\Zbb $$
coming from the local invariant map 
$$ \textup{inv}_{p} : \textup{Br}\,\Qpbb \rightarrow \Qbb/\Zbb $$ from local class field theory (this is an isomorphism if $p$ is a prime number). This pairing is locally constant on the left by \cite{Poo17}, Proposition 8.2.9. Any element $\mathcal{B} \in \textup{Br}\,U$ pairs trivially on $U(\Qpbb)$ for almost all $p$, thus taking the sum of the local pairings gives a pairing
$$ \prod_{p} U(\Qpbb) \times \textup{Br}\,U \rightarrow \Qbb/\Zbb. $$
This factors through the group $\textup{Br}\,U/\textup{Br}\,\Qbb$ and pairs trivially with the elements of $U(\Qbb)$. For $B \subseteq \textup{Br}\,U$, let $(\prod_{p} U(\Qpbb))^{B}$ be the left kernel of this pairing with respect to $B$. By Theorem 3.8, the group $\textup{Br}_{1}\,U/\textup{Br}\,\Qbb$ is generated by the algebra $\mathcal{A}$. Thus in our case, it suffices to consider the sequence of inclusions $U(\Qbb) \subseteq (\prod_{p} U(\Qpbb))^{\mathcal{A}} \subseteq \prod_{p} U(\Qpbb)$. In particular, if the latter inclusion is strict, then $\mathcal{A}$ gives an obstruction to \textit{weak approximation} on $U$.

For \textbf{integral points}, any element $\mathcal{B} \in \textup{Br}\,U$ pairs trivially on $\mathcal{U}(\Zpbb)$ for almost all $p$, so we obtain a pairing $U(\textbf{\textup{A}}_{\Qbb}) \times \textup{Br}\,U \rightarrow \Qbb/\Zbb$. As the local pairings are locally constant, we obtain a well-defined pairing (by abuse of notation, we write $\mathcal{U}(\textbf{\textup{A}}_{\Zbb})_{\bullet}$ the same as $\mathcal{U}(\textbf{\textup{A}}_{\Zbb})$): 
$$ \mathcal{U}(\textbf{\textup{A}}_{\Zbb}) \times \textup{Br}\,U \rightarrow \Qbb/\Zbb. $$
For $B \subseteq \textup{Br}\,U$, let $\mathcal{U}(\textbf{\textup{A}}_{\Zbb})^{B}$ be the left kernel with respect to $B$, and let $\mathcal{U}(\textbf{\textup{A}}_{\Zbb})^{\textup{Br}} = \mathcal{U}(\textbf{\textup{A}}_{\Zbb})^{\textup{Br}\,U}$. By Theorem 3.8, the map $\langle \mathcal{A} \rangle \rightarrow \textup{Br}_{1}\,U/\textup{Br}\,\Qbb$ is an isomorphism, hence $\mathcal{U}(\textbf{\textup{A}}_{\Zbb})^{\textup{Br}_{1}\,U} = \mathcal{U}(\textbf{\textup{A}}_{\Zbb})^{\mathcal{A}}$. We have the inclusions $\mathcal{U}(\Zbb) \subseteq \mathcal{U}(\textbf{\textup{A}}_{\Zbb})^{\mathcal{A}} \subseteq \mathcal{U}(\textbf{\textup{A}}_{\Zbb})$ so that $\mathcal{A}$ can obstruct the \textit{integral Hasse principle} or \textit{strong approximation}.

Let $V$ be dense Zariski open in $U$. As $U$ is smooth, the set $V(\Qpbb)$ is dense in $U(\Qpbb)$ for all places $p$. Moreover, $\mathcal{U}(\Zpbb)$ is open in $U(\Qpbb)$, hence $V(\Qpbb) \cap \mathcal{U}(\Zpbb)$ is dense in $\mathcal{U}(\Zpbb)$. As the local pairings are locally constant, we may restrict our attention to $V$ to calculate the local invariants of a given element in $\textup{Br}\,U$.

\subsection{Brauer--Manin obstruction from quaternion algebras}
Now we consider the \textbf{three} explicit families of Markoff-type K3 (MK3) surfaces over $\Qbb$ as introduced before. From now on, we always denote by $W_{k}$ the projective MK3 surfaces, $U_{k}$ the affine open subscheme defined by $W_{k} \setminus \{rst = 0\}$ and $\mathcal{U}_{k}$ the integral model of $U_{k}$ defined by the same equation.

\subsubsection{Existence of local points}
First of all, we study the existence of local integral points on the affine MK3 surfaces. It is interesting to note that there always exist $\Qbb$-points at infinity (when $rst = 0$) on these surfaces.

\begin{proposition}[Assumption I]
For $k \in \Zbb$, let $W_{k} \subset \Pbb^{1} \times \Pbb^{1} \times \Pbb^{1}$ be the MK3 surfaces defined over $\Qbb$ by the $(2,2,2)$-form 
\begin{equation}
    F_{1}(x,y,z) = x^{2} + y^{2} + z^{2} - 4x^{2}y^{2}z^{2} - k = 0.
\end{equation}
Let $\mathcal{U}_{k}$ be the integral model of $U_{k}$ defined over $\Zbb$ by the same equation. If $k$ satisfies the conditions:
\begin{enumerate}
    \item $k \equiv -1$ \textup{mod} $8$;
    \item $k \not \equiv 0$ \textup{mod} $3, 5, 7$,
\end{enumerate}
then $\mathcal{U}_{k}(\textbf{\textup{A}}_{\Zbb}) \not= \emptyset$.
\end{proposition}

\begin{proof}
For the place at infinity, it is clear that there exist real solutions: If $k \geq 0$ then take $(x,y,z) = (\sqrt{k},0,0)$; if $k \leq -1$ then take $x = y = z \geq 1$ which satisfies $3x^{2} - 4x^{6} = k$ as the left hand side is a strictly decreasing continuous function of value $\leq -1$ on $[1,+\infty)$. For solutions at finite places $p$, with our specific conditions for $k$ in the assumption, we have:
\begin{enumerate}
	\item[(i)] Prime powers of $p = 2$: It is clear that every solution modulo $2$ is singular. Thanks to the condition (1), we find the non-singular solution $(1,1,1)$ modulo $8$, which then lifts to solutions modulo higher powers of $2$ by Hensel's lemma.
	
	\item[(ii)] Prime powers of $p \geq 3$: We need to find a non-singular solution modulo $p$ of the equations $F_{1} = 0$ which does \textbf{not} satisfy simultaneously
	$$ dF_{1}=0 : 2x(1-4y^{2}z^{2}) = 0, 2y(1-4z^{2}x^{2}) = 0, 2z(1-4x^{2}y^{2}) = 0. $$
	For simplicity, we will try to find a non-singular solution whose $z = 0$. First, it is clear that the equation $F_{1}=0$ always has a solution when $z = 0$: indeed, take $z=0$, then $F_{1}=0$ becomes $x^{2} + y^{2} = 0$, and every element in $\Fpbb$ can be expressed as the sum of two squares. Note that such a solution is singular if and only if $x = y =0$, which means that $p$ divides $k$. Hence, if $p$ does not divide $k$, then we can find a non-singular solution mod $p$ which lifts to higher powers of $p$ by Hensel's lemma. In particular, this is true for $p = 3, 5, 7$ thanks to the condition (2).
	\\~\\
	\indent Next, consider the case when $p \geq 11$ and $p$ divides $k$. We will find instead a non-singular solution whose $z=1$. The equation becomes
	$$ F_{1}(x,y,1) = x^{2} + y^{2} - 4x^{2}y^{2} + (1-k) = 0 $$
	which defines an affine curve $C \subset \Abb^{2}_{(x,y)}$ over $\Qbb$. If we consider its projective closure in $\Pbb^{2}_{[x:y:t]}$ defined by
	$$ t^{2}(x^{2}+y^{2}) - 4x^{2}y^{2} + (1-k)t^{4} = 0, $$
	then we can see that it has only two singularities which are \emph{ordinary} of multiplicity $2$, namely $[1:0:0]$ and $[0:1:0]$. By the genus--degree formula and the fact that the geometric genus is a birational invariant, we obtain
	$$ g(C) = \frac{(\deg C - 1)(\deg C - 2)}{2} - \sum_{i = 1}^{n} \frac{r_{i}(r_{i}-1)}{2}, $$
	where $n$ the number of ordinary singularities and $r_{i}$ is the multiplicity of each singularity for $i = 1,\dots,n$; in particular, $g(C) = 3 - 2 = 1$.
	
	Now we consider the original projective closure $C^{1} \in \Pbb^{1}_{[x:r]} \times \Pbb^{1}_{[y:s]}$ defined by
	$$ x^{2}s^{2} + y^{2}r^{2} - 4x^{2}y^{2} + (1-k)r^{2}s^{2} = 0. $$
	The projective curve $C^{1}$ is \textbf{smooth} over $\Fpbb$ under our assumption on $k$. Then by the Hasse--Weil bound for smooth, projective and geometrically integral curves of genus $1$, we have
	$$ |C^{1}(\Fpbb)| \geq p + 1 - 2\sqrt{p} = (\sqrt{p} - 1)^{2} > (3-1)^{2} = 4 $$
	since $p \geq 11$, so $|C^{1}(\Fpbb)| \geq 5$. As $C^{1}$ has exactly $4$ points at infinity (when $rs = 0$), the affine curve $C$ has at least one \emph{smooth} $\Fpbb$-point which then lifts to higher powers of $p$ by Hensel's lemma.
\end{enumerate}
\end{proof}

\begin{proposition}[Assumption II]
For $k \in \Zbb$, let $W_{k} \subset \Pbb^{1} \times \Pbb^{1} \times \Pbb^{1}$ be the MK3 surfaces defined over $\Qbb$ by the $(2,2,2)$-form 
\begin{equation}
    F_{2}(x,y,z) = x^{2} + y^{2} + z^{2} - 4(x^{2}y^{2} + y^{2}z^{2} + z^{2}x^{2}) + 16x^{2}y^{2}z^{2} - k = 0.
\end{equation}
Let $\mathcal{U}_{k}$ be the integral model of $U_{k}$ defined over $\Zbb$ by the same equation. If $k$ satisfies the conditions:
\begin{enumerate}
    \item $k \equiv 2$ \textup{mod} $8$, $k \equiv -9$ \textup{mod} $27$, $k \equiv -2$ \textup{mod} $5$, and $k \equiv 2$ \textup{mod} $7$;
    \item $p \equiv \pm 1$ \textup{mod} $8$ for any odd prime divisor $p$ of $k$,
\end{enumerate}
then $\mathcal{U}_{k}(\textbf{\textup{A}}_{\Zbb}) \not= \emptyset$.
\end{proposition}

\begin{proof}
For the place at infinity, it is clear that there exist real solutions: If $k \geq 0$ then take $(x,y,z) = (\sqrt{k},0,0)$; if $k \leq -1$ then take $y = 1, z = 0$ and $x = \sqrt{\frac{1-k}{3}}$. For solutions at finite places $p$, with our specific conditions for $k$ in the assumption, we have:
\begin{enumerate}
	\item[(i)] Prime powers of $p = 2$: It is clear that every solution modulo $2$ is singular. Thanks to the condition (1), we find the non-singular solution $(1,1,2)$ modulo $8$, which then lifts to solutions modulo higher powers of $2$ by Hensel's lemma.
	
	\item[(ii)] Prime powers of $p = 3,5$: Thanks to the condition (1), we find the non-singular solutions $(3,3,0)$ modulo $27$ and $(1,1,0)$ modulo $5$, which lifts to higher powers of $3$ and $5$ by Hensel's lemma.
	
	\item[(iii)] Prime powers of $p \geq 7$: We need to find a non-singular solution modulo $p$ of the equations $F_{2} = 0$ which does \textbf{not} satisfy simultaneously 
	$$ dF_{2}=0 : 2x(1-4y^{2})(1-4z^{2}) = 0, 2y(1-4z^{2})(1-4x^{2}) = 0, 2z(1-4x^{2})(1-4y^{2}) = 0. $$
	First, note that the equation $F_{2}=0$ is equivalent to 
	$$ (4x^{2}-1)(4y^{2}-1)(4z^{2}-1) = 4k-1. $$
	We observe that there are two special cases: if $p$ divides $k$ then there exists a non-singular solution $(a,b,0)$ where $2a^{2} \equiv 2b^{2} \equiv 1$ mod $p$ thanks to the condition (2); if $p$ divides $4k-1$ then there clearly exists a non-singular solution $(\frac{1}{2},0,0)$. In particular, this is true for $p = 7$ thanks to the condition (1).	
	\\~\\
	\indent Next, consider the case when $p \geq 11$ and $p$ does not divide either $k$ or $4k-1$. For simplicity, we will try to find a non-singular solution whose $z = 0$. The equation becomes
	$$ F_{2}(x,y,0) = x^{2} + y^{2} - 4x^{2}y^{2} - k = 0 $$
	which defines an affine curve $C \subset \Abb^{2}_{(x,y)}$ over $\Qbb$. If we consider its projective closure in $\Pbb^{2}_{[x:y:t]}$ defined by
	$$ t^{2}(x^{2}+y^{2}) - 4x^{2}y^{2} - kt^{4} = 0, $$
	then we can see that it has only two singularities which are \emph{ordinary} of multiplicity $2$, namely $[1:0:0]$ and $[0:1:0]$. By the genus--degree formula and the fact that the geometric genus is a birational invariant, we obtain
	$$ g(C) = \frac{(\deg C - 1)(\deg C - 2)}{2} - \sum_{i = 1}^{n} \frac{r_{i}(r_{i}-1)}{2}, $$
	where $n$ the number of ordinary singularities and $r_{i}$ is the multiplicity of each singularity for $i = 1,\dots,n$; in particular, $g(C) = 3 - 2 = 1$.
	
	Now we consider the original projective closure $C^{1} \in \Pbb^{1}_{[x:r]} \times \Pbb^{1}_{[y:s]}$ defined by
	$$ x^{2}s^{2} + y^{2}r^{2} - 4x^{2}y^{2} - kr^{2}s^{2} = 0. $$
	The projective curve $C^{1}$ is \textbf{smooth} over $\Fpbb$ under our assumption on $k$. Then by the Hasse--Weil bound for smooth, projective and geometrically integral curves of genus $1$, we have
	$$ |C^{1}(\Fpbb)| \geq p + 1 - 2\sqrt{p} = (\sqrt{p} - 1)^{2} > (3-1)^{2} = 4 $$
	since $p \geq 11$, so $|C^{1}(\Fpbb)| \geq 5$. As $C^{1}$ has exactly $4$ points at infinity (when $rs = 0$), the affine curve $C$ has at least one \emph{smooth} $\Fpbb$-point which then lifts to higher powers of $p$ by Hensel's lemma.
\end{enumerate}
\end{proof}

\begin{proposition}[Assumption III]
For $k \in \Zbb$, let $W_{k} \subset \Pbb^{1} \times \Pbb^{1} \times \Pbb^{1}$ be the MK3 surfaces defined over $\Qbb$ by the $(2,2,2)$-form 
\begin{equation}
    F_{3}(x,y,z) = x^{2} + y^{2} + z^{2} + 4(x^{2}y^{2} + y^{2}z^{2} + z^{2}x^{2}) - 16x^{2}y^{2}z^{2} - k = 0.
\end{equation}
Let $\mathcal{U}_{k}$ be the integral model of $U_{k}$ defined over $\Zbb$ by the same equation. If $k$ satisfies the conditions:
\begin{enumerate}
    \item $k \equiv 1$ \textup{mod} $4$, $k \equiv 2$ \textup{mod} $3$, $k \equiv 3$ \textup{mod} $5$;
    \item $k \not\equiv 0,-2$ \textup{mod} $7$ and $k \not \equiv 0$ \textup{mod} $37$,
\end{enumerate}
then $\mathcal{U}_{k}(\textbf{\textup{A}}_{\Zbb}) \not= \emptyset$.
\end{proposition}

\begin{proof}
For the place at infinity, it is clear that there exist real solutions: If $k \geq 0$ then take $(x,y,z) = (\sqrt{k},0,0)$; if $k \leq -1$ then take $x = y = z \geq 1$ which satisfies $3x^{2} + 12x^{4} - 16x^{6} = k$ as the left hand side is a strictly decreasing continuous function of value $\leq -1$ on $[1,+\infty)$. For solutions at finite places $p$, with our specific conditions for $k$ in the assumption, we have:
\begin{enumerate}
	\item[(i)] Prime powers of $p = 2$: It is clear that every solution modulo $2$ is singular. Thanks to the condition (1), we find the non-singular solutions $(1,0,0)$ mod $8$ if $k \equiv 1$ mod $8$ and $(1,2,0)$ mod $8$ if $k \equiv 5$ mod $8$, which then lift to solutions modulo higher powers of $2$ by Hensel's lemma.
	
	\item[(ii)] Prime powers of $p = 3, 5$: Thanks to the condition (1), we find the non-singular solutions $(1,1,1)$ for $p = 3$ if $k \equiv 2$ mod $3$ and $(1,2,1)$ for $p = 5$ if $k \equiv 3$ mod $5$, which then lift to solutions modulo higher powers of $3$ and $5$ by Hensel's lemma.
	
	\item[(iii)] Prime powers of $p \geq 7$: We need to find a non-singular solution modulo $p$ of the equations $F_{3} = 0$ which does \textbf{not} satisfy simultaneously 
	$$ dF_{3}=0 : 2x(1-4y^{2})(1-4z^{2}) = 0, 2y(1-4z^{2})(1-4x^{2}) = 0, 2z(1-4x^{2})(1-4y^{2}) = 0. $$
	For simplicity, we will try to find a non-singular solution whose $z = 0$. The equation becomes
	$$ F_{3}(x,y,0) = x^{2} + y^{2} + 4x^{2}y^{2} - k = 0 $$
	which defines an affine curve $C \subset \Abb^{2}_{(x,y)}$ over $\Qbb$. If we consider its projective closure in $\Pbb^{2}_{[x:y:t]}$ defined by
	$$ t^{2}(x^{2}+y^{2}) + 4x^{2}y^{2} - kt^{4} = 0, $$
	then we can see that it has only two singularities which are \emph{ordinary} of multiplicity $2$, namely $[1:0:0]$ and $[0:1:0]$. By the genus--degree formula and the fact that the geometric genus is a birational invariant, we obtain
	$$ g(C) = \frac{(\deg C - 1)(\deg C - 2)}{2} - \sum_{i = 1}^{n} \frac{r_{i}(r_{i}-1)}{2}, $$
	where $n$ the number of ordinary singularities and $r_{i}$ is the multiplicity of each singularity for $i = 1,\dots,n$; in particular, $g(C) = 3 - 2 = 1$.
	
	Now we consider the original projective closure $C^{1} \in \Pbb^{1}_{[x:r]} \times \Pbb^{1}_{[y:s]}$ defined by
	$$ x^{2}s^{2} + y^{2}r^{2} + 4x^{2}y^{2} - kr^{2}s^{2} = 0. $$
	If $p$ does not divide either $k$ or $4k+1$, then the projective curve $C^{1}$ is \textbf{smooth} over $\Fpbb$ under our assumption on $k$. Then by the Hasse--Weil bound for smooth, projective and geometrically integral curves of genus $1$, we have
	$$ |C^{1}(\Fpbb)| \geq p + 1 - 2\sqrt{p} = (\sqrt{p} - 1)^{2}, $$
	so $|C^{1}(\Fpbb)| \geq 3$ if $p = 7$ and $|C^{1}(\Fpbb)| \geq 5$ if $p \geq 11$. As $C^{1}$ has exactly $2$ and $4$ points at infinity (when $rs = 0$) if $p \equiv 3$ and $1$ mod $4$ respectively, the affine curve $C$ has at least one \emph{smooth} $\Fpbb$-point which then lifts to higher powers of $p$ by Hensel's lemma. In particular, this is true for $p = 7$ thanks to the condition (2).
	\\~\\
	\indent Next, consider the case when $p \geq 11$ and $p$ divides $k$ or $4k+1$. We will find instead a non-singular solution whose $z=1$. The equation becomes
	$$ F_{3}(x,y,1) = 5x^{2} + 5y^{2} - 12x^{2}y^{2} + (1-k) = 0 $$
	which defines an affine curve $D \subset \Abb^{2}_{(x,y)}$ over $\Qbb$. If we consider its projective closure in $\Pbb^{2}_{[x:y:t]}$ defined by
	$$ 5t^{2}(x^{2}+y^{2}) - 12x^{2}y^{2} + (1-k)t^{4} = 0, $$
	then we can see that it also has only two singularities which are \emph{ordinary} of multiplicity $2$, namely $[1:0:0]$ and $[0:1:0]$. By the genus--degree formula and the fact that the geometric genus is a birational invariant, we obtain
	$$ g(D) = \frac{(\deg D - 1)(\deg D - 2)}{2} - \sum_{i = 1}^{n} \frac{r_{i}(r_{i}-1)}{2}, $$
	where $n$ the number of ordinary singularities and $r_{i}$ is the multiplicity of each singularity for $i = 1,\dots,n$; in particular, $g(D) = 3 - 2 = 1$.
	
	Now we consider the original projective closure $D^{1} \in \Pbb^{1}_{[x:r]} \times \Pbb^{1}_{[y:s]}$ defined by
	$$ 5x^{2}s^{2} + 5y^{2}r^{2} - 12x^{2}y^{2} + (1-k)r^{2}s^{2} = 0. $$
	The projective curve $D^{1}$ is \textbf{smooth} over $\Fpbb$ under our assumption on $k$, especially the additional hypothesis $k \not \equiv 0$ mod $37$. Then by the Hasse--Weil bound for smooth, projective and geometrically integral curves of genus $1$, we have $|D^{1}(\Fpbb)| \geq 5$ since $p \geq 11$. As $D^{1}$ has at most $4$ points at infinity (when $rs = 0$), the affine curve $D$ has at least one \emph{smooth} $\Fpbb$-point which then lifts to higher powers of $p$ by Hensel's lemma. The proof is now complete.
\end{enumerate}
\end{proof}

\subsubsection{Integral Brauer--Manin obstructions}
It is important to recall that there always exist $\Qbb$-points (at infinity) on every member $W_{k} \subset \Pbb^{1} \times \Pbb^{1} \times \Pbb^{1}$ of each family of these Markoff-type K3 surfaces, hence they satisfy the (rational) Hasse principle. Now we prove the Brauer--Manin obstructions to the integral Hasse principle on the integral model $\mathcal{U}_{k}$ of the affine subscheme $U_{k} \subset W_{k}$ by calculating the local invariants for some quaternion algebra classes $\mathcal{A}$ in their Brauer groups:
$$ \textup{inv}_{p}\,\mathcal{A} : \mathcal{U}_{k}(\Zpbb) \rightarrow \Zbb/2\Zbb, \hspace{1cm} x \mapsto \textup{inv}_{p}\,\mathcal{A}(x). $$

\begin{theorem}
For $k \in \Zbb$, let $W_{k} \subset \Pbb^{1} \times \Pbb^{1} \times \Pbb^{1}$ be the MK3 surfaces defined over $\Qbb$ by the $(2,2,2)$-form 
\begin{equation}
    F_{1}(x,y,z) = x^{2} + y^{2} + z^{2} - 4x^{2}y^{2}z^{2} - k = 0.
\end{equation}
Let $\mathcal{U}_{k}$ be the integral model of $U_{k}$ defined over $\Zbb$ by the same equation. If $k$ satisfies the conditions:
\begin{enumerate}
    \item $k = -(1 + 16\ell^{2})$ where $\ell \in \Zbb$ such that $\ell$ is odd and $\ell \not\equiv \pm 2$ \textup{mod} $5$;
    \item $p \equiv 1$ \textup{mod} $4$ for any prime divisor $p$ of $\ell$,
\end{enumerate}
then there is an algebraic Brauer--Manin obstruction to the integral Hasse principle on $\mathcal{U}_{k}$ with respect to the element $\mathcal{A} = (4x^{2}y^{2}-1, k+1) = (4y^{2}z^{2}-1, k+1) = (4z^{2}x^{2}-1, k+1)$ in $\textup{Br}_{1}\,U_{k}/\textup{Br}_{0}\,U_{k}$. In other words, $\mathcal{U}_{k}(\Zbb) \subset \mathcal{U}_{k}(\textbf{\textup{A}}_{\Zbb})^{\mathcal{A}} = \emptyset$.
\end{theorem}

\begin{proof}
For any local point in $\mathcal{U}_{k}(\textbf{A}_{\Zbb})$, we calculate its local invariants at every prime $p \leq \infty$. First of all, note that $U_{k}$ is smooth over $\Qbb$ and the affine equation implies
$$ (4x^{2}y^{2}-1)(4y^{2}z^{2}-1) = (2y^{2}+1)^{2} - 4(k+1)y^{2}. $$ 
Therefore, we obtain the equality $$\mathcal{A} = (4x^{2}y^{2}-1, k+1) = (4y^{2}z^{2}-1, k+1) = (4z^{2}x^{2}-1, k+1)$$ in $\textup{Br}_{1}\,U_{k}/\textup{Br}_{0}\,U_{k}$. Now by abuse of notation, at each place $p$ we consider a local point denoted by $(x,y,z)$.

At $p = \infty$: From the equation $z^{2}(4x^{2}y^{2}-1) = x^{2}+y^{2}-k > 0$ for all $x,y,z \in \Rbb$ since $k \leq -1$ by our assumption, so $4x^{2}y^{2}-1 > 0$ for every point $(x,y,z) \in U_{k}(\Rbb)$. Hence we have $\textup{inv}_{\infty}\,\mathcal{A}(x,y,z) = 0$.

At $p = 2$: Since $k \equiv -1$ mod $8$, all the coordinates $x,y,z$ are in $\Zbb_{2}^{\times}$, then $4x^{2}y^{2} - 1 \equiv 3$ mod $8$ so $\textup{inv}_{2}\,\mathcal{A}(x,y,z) = (4x^{2}y^{2}-1,k+1)_{2} = (3,-1)_{2} = \frac{1}{2}$.

At $p \geq 3$: Since $k+1 = -16\ell^{2}$ and every odd prime divisor $p$ of it satisfies $(-1,p)_{p} = 0$, if $p$ divides $k+1$ then $\textup{inv}_{p}\,\mathcal{A}(x,y,z) = 0$. Otherwise, if $p$ divides $4x^{2}y^{2}-1$ then $p$ cannot divide $y$ and so by the above equation we have $k+1 \in \Zpbb^{\times 2}$, which implies that $(4x^{2}y^{2}-1,k+1)_{p} = 0$. Finally, if $4x^{2}y^{2}-1$ and $k+1$ are both in $\Zpbb^{\times}$ then the local invariant is trivial as well.

In conclusion, we have 
$$ \sum_{p \leq \infty} \textup{inv}_{p}\,\mathcal{A}(x,y,z) = \frac{1}{2} \not= 0, $$
so $\mathcal{U}_{k}(\Zbb) \subset \mathcal{U}_{k}(\textbf{\textup{A}}_{\Zbb})^{\mathcal{A}} = \emptyset$.
\end{proof}

\begin{theorem}
For $k \in \Zbb$, let $W_{k} \subset \Pbb^{1} \times \Pbb^{1} \times \Pbb^{1}$ be the MK3 surfaces defined over $\Qbb$ by the $(2,2,2)$-form 
\begin{equation}
    F_{2}(x,y,z) = x^{2} + y^{2} + z^{2} - 4(x^{2}y^{2} + y^{2}z^{2} + z^{2}x^{2}) + 16x^{2}y^{2}z^{2} - k = 0.
\end{equation}
Let $\mathcal{U}_{k}$ be the integral model of $U_{k}$ defined over $\Zbb$ by the same equation. If $k$ satisfies the conditions:
\begin{enumerate}
    \item $k = 18\ell^{2}$ where $\ell \in \Zbb$ such that $\ell \not\equiv 0$ \textup{mod} $2,3$, $\ell \equiv 1$ \textup{mod} $5$, and $\ell \equiv 2$ \textup{mod} $7$;
    \item $p \equiv \pm 1$ \textup{mod} $8$ for any prime divisor $p$ of $\ell$,
\end{enumerate}
then there is an algebraic Brauer--Manin obstruction to the integral Hasse principle on $\mathcal{U}_{k}$ with respect to the subgroup $A \subset \textup{Br}_{1}\,U_{k}/\textup{Br}_{0}\,U_{k}$ generated by the elements $\mathcal{A}_{1} = (4x^{2}-1, k)$ and $\mathcal{A}_{2} = (4y^{2}-1, k)$, i.e., $\mathcal{U}_{k}(\Zbb) \subset \mathcal{U}_{k}(\textbf{\textup{A}}_{\Zbb})^{A} = \emptyset$.
\end{theorem}

\begin{proof}
For any local point in $\mathcal{U}_{k}(\textbf{A}_{\Zbb})$, we calculate its local invariants at every prime $p \leq \infty$. First of all, note that $U_{k}$ is smooth over $\Qbb$ and the affine equation implies 
$$ (4x^{2}-1)(4y^{2}-1)(4z^{2}-1) = 4k-1. $$ 
Therefore, we obtain the equality $$(4x^{2}-1, k) + (4y^{2}-1, k) + (4z^{2}-1, k) = 0$$ in $\textup{Br}_{1}\,U_{k}/\textup{Br}_{0}\,U_{k}$. Now by abuse of notation, at each place $p$ we consider a local point denoted by $(x,y,z)$.

At $p = \infty$: For any $\ell \in \Rbb$, we always have $k = 18\ell^{2} > 0$, hence $\textup{inv}_{\infty}\,\mathcal{A}(x,y,z) = 0$.

At $p = 2$: Since $k \equiv 2$ mod $8$, exactly two of the coordinates $x,y,z$ are in $\Zbb_{2}^{\times}$, so without loss of generality let one of them be $x$, then $4x^{2} - 1 \equiv 3$ mod $8$ so $\textup{inv}_{2}\,\mathcal{A}(x,y,z) = (4x^{2}-1,k)_{2} = (3,2)_{2} = \frac{1}{2}$.

At $p = 3$: Since $k = 18\ell^{2}$, all of the coordinates $x,y,z$ must be divisible by $3$, so $\textup{inv}_{3}\,\mathcal{A}(x,y,z) = (-1,18)_{3} = 0$. 

At $p \geq 5$: Since $k = 18\ell^{2}$ and every odd prime divisor $p \not= 3$ satisfies $(2,p)_{p} = 0$, if $p$ divides $k$ then $\textup{inv}_{p}\,\mathcal{A}(x,y,z) = 0$. Otherwise, if $p$ divides $4x^{2}-1$ then by the above equation we have $k \in \Zpbb^{\times 2}$, which implies that $(4x^{2}-1,k)_{p} = 0$. Finally, if $4x^{2}-1$ and $k$ are both in $\Zpbb^{\times}$ then the local invariant is trivial as well.

In conclusion, we have 
$$ \sum_{p \leq \infty} \textup{inv}_{p}\,\mathcal{A}(x,y,z) = \frac{1}{2} \not= 0, $$
so $\mathcal{U}_{k}(\Zbb) \subset \mathcal{U}_{k}(\textbf{\textup{A}}_{\Zbb})^{\mathcal{A}} = \emptyset$.
\end{proof}

\begin{theorem}
For $k \in \Zbb$, let $W_{k} \subset \Pbb^{1} \times \Pbb^{1} \times \Pbb^{1}$ be the MK3 surfaces defined over $\Qbb$ by the $(2,2,2)$-form 
\begin{equation}
    F_{3}(x,y,z) = x^{2} + y^{2} + z^{2} + 4(x^{2}y^{2} + y^{2}z^{2} + z^{2}x^{2}) - 16x^{2}y^{2}z^{2} - k = 0.
\end{equation}
Let $\mathcal{U}_{k}$ be the integral model of $U_{k}$ defined over $\Zbb$ by the same equation. If $k$ satisfies the conditions:
\begin{enumerate}
    \item $k = -\frac{1}{4}(1 + 27\ell^{2})$ where $\ell \in \Zbb$ such that $\ell \equiv \pm 1$ \textup{mod} $8$, $\ell \equiv 1$ \textup{mod} $5$, $\ell \equiv 3$ \textup{mod} $7$, and $\ell \not\equiv \pm 10$ \textup{mod} $37$;
    \item $p \equiv \pm 1$ \textup{mod} $24$ for any prime divisor $p$ of $\ell$,
\end{enumerate}
then there is an algebraic Brauer--Manin obstruction to the integral Hasse principle on $\mathcal{U}_{k}$ with respect to the subgroup $A \subset \textup{Br}_{1}\,U_{k}/\textup{Br}_{0}\,U_{k}$ generated by the elements $\mathcal{A}_{1} = (4x^{2}+1, -2(4k+1))$ and $\mathcal{A}_{2} = (4y^{2}+1, -2(4k+1))$, i.e., $\mathcal{U}_{k}(\Zbb) \subset \mathcal{U}_{k}(\textbf{\textup{A}}_{\Zbb})^{A} = \emptyset$.
\end{theorem}

\begin{proof}
For any local point in $\mathcal{U}_{k}(\textbf{A}_{\Zbb})$, we calculate its local invariants at every prime $p \leq \infty$. First of all, note that $U_{k}$ is smooth over $\Qbb$ and the affine equation implies 
$$ (4x^{2}+1)(4y^{2}+1)(4z^{2}+1) = (4k+1) + 128x^{2}y^{2}z^{2}. $$ 
Therefore, we obtain the equality $$(4x^{2}+1, -2(4k+1)) + (4y^{2}+1, -2(4k+1)) + (4z^{2}+1, -2(4k+1)) = 0$$ in $\textup{Br}_{1}\,U_{k}/\textup{Br}_{0}\,U_{k}$. Now by abuse of notation, at each place $p$ we consider a local point denoted by $(x,y,z)$.

At $p = \infty$: For any $x \in \Rbb$, we always have $4x^{2}+1 > 0$, hence $\textup{inv}_{\infty}\,\mathcal{A}(x,y,z) = 0$.

At $p = 2$: Since $k \equiv 1$ mod $4$, exactly two of the coordinates $x,y,z$ are in $2\Zbb_{2}$, so without loss of generality let one of them be $x$, then $4x^{2} + 1 \equiv 1$ mod $8$ so $\textup{inv}_{2}\,\mathcal{A}(x,y,z) = (4x^{2}+1,-2(4k+1))_{2} = 0$.

At $p = 3$: Since $k \equiv 2$ mod $3$, all of the coordinates $x,y,z$ are in $\Zbb_{3}^{\times}$, so $\textup{inv}_{3}\,\mathcal{A}(x,y,z) = (2,54\ell^{2})_{3} = (-1,3)_{3} = \frac{1}{2}$. 

At $p \geq 5$: Since $-2(4k+1) = 54\ell^{2}$ and every odd prime divisor $p \not= 3$ satisfies $(6,p)_{p} = 0$, if $p$ divides $4k+1$ then $\textup{inv}_{p}\,\mathcal{A}(x,y,z) = 0$. Otherwise, if $p$ divides $4x^{2}+1$ then by the above equation we have $-2(4k+1) \in \Zpbb^{\times 2}$, which implies that $(4x^{2}+1,-2(4k+1))_{p} = 0$. Finally, if $4x^{2}+1$ and $4k+1$ are both in $\Zpbb^{\times}$ then the local invariant is trivial as well.

In conclusion, we have 
$$ \sum_{p \leq \infty} \textup{inv}_{p}\,\mathcal{A}(x,y,z) = \frac{1}{2} \not= 0, $$
so $\mathcal{U}_{k}(\Zbb) \subset \mathcal{U}_{k}(\textbf{\textup{A}}_{\Zbb})^{\mathcal{A}} = \emptyset$.
\end{proof}

\begin{example}
We give some explicit counterexamples to the integral Hasse principle for the three families of Markoff-type K3 surfaces that we have discussed. Note that in theory, there always exist \textbf{primes} $\ell$ which satisfy all the hypotheses for each family, thanks to the well-known Dirichlet's theorem on arithmetic progressions.
\begin{enumerate}
	\item[(1)] For $\ell = 1$, we have
	$$ x^{2} + y^{2} + z^{2} - 4x^{2}y^{2}z^{2} = -(1+16.1^{2}) = -17. $$
	
	\item[(2)] For $\ell = 191$, we have
	$$ x^{2} + y^{2} + z^{2} - 4(x^{2}y^{2} + z^{2}x^{2} + x^{2}y^{2}) + 16x^{2}y^{2}z^{2} = 18.191^2 = 656658. $$
	
	\item[(3)] For $\ell = 241$, we have
	$$ x^{2} + y^{2} + z^{2} + 4(x^{2}y^{2} + z^{2}x^{2} + x^{2}y^{2}) - 16x^{2}y^{2}z^{2} = -\frac{1}{4}(1+27.241^{2}) = -392047. $$
\end{enumerate}
\end{example}

\begin{remark}
Note that in our first and third cases, we need $k$ to be negative. In fact, if $k = 0$ then we always have the trivial solution $(0,0,0)$. And if $k \geq 0$ and $k$ satisfies our assumption for each family of those Markoff-type K3 surfaces, then we can prove the nonexistence of integral points via elementary arguments.
\end{remark}

\subsection{Counting the Hasse failures}
In this part, we calculate the number of examples of existence for local integral points as well as the number of counterexamples to the integral Hasse principle for our Markoff-type K3 surfaces which can be explained by the Brauer--Manin obstruction. More precisely, we compute the natural density of $k \in \Zbb$ satisfying the hypotheses in Assumptions I, II, III and the three main Theorems about the Brauer--Manin obstruction.

\begin{theorem}
For the above three families of MK3 surfaces, we have
	$$ \# \{k \in \Zbb: |k| \leq M,\ \mathcal{U}_{k}(\textbf{\textup{A}}_{\Zbb}) \not= \emptyset \} \asymp M $$
	and
	$$ \# \{k \in \Zbb: |k| \leq M,\ \mathcal{U}_{k}(\textbf{\textup{A}}_{\Zbb}) \not= \emptyset,\ \mathcal{U}_{k}(\textbf{\textup{A}}_{\Zbb})^{\textup{Br}} = \emptyset \} \gg \frac{M^{1/2}}{\textup{log}\,M}, $$
as $M \rightarrow +\infty$.
\end{theorem}

\begin{proof}
For the first approximation, the result follows directly from the fact that Assumptions I, II, III only give finitely many congruence conditions on $k$, so the total numbers of $k$ are always a proportion of $M$.

For the second approximation, we only give an asymptotic lower bound with the condition that $\ell$ is a \textbf{prime}. The result follows from the fact that as $M \rightarrow +\infty$, $|k|$ is approximately a multiple of $\ell^{2}$, and the number of primes less than $\sqrt{N}$ (here $N$ is a proportion of $M$) satisfying finitely many congruence conditions is asymptotically equal to $\ds\frac{\sqrt{M}}{\log M}$, up to a constant factor (see \cite{Apt76}, Section 7.9).
\end{proof}

\begin{remark}
Continuing from a previous remark, it would be interesting if one can find a way to include the transcendental Brauer group into the counting result, which would help us consider the Brauer--Manin set with respect to the whole Brauer group instead of only its algebraic part.
\end{remark}

\section{Further remarks}
In this section, we compare the results that we obtain in this paper with those in the previous papers studying Markoff surfaces, namely \cite{GS22}, \cite{LM20}, \cite{CTWX20}, and \cite{Dao22}. 

\subsection{Existence of the Brauer--Manin obstruction}
First of all, recall that in the case of Markoff surfaces, we see from \cite{LM20} that the number of counterexamples to the integral Hasse principle which can be explained by the Brauer--Manin obstruction is asymptotically equal to $M^{1/2}/(\log M)^{1/2}$; this number is also the asymptotic lower bound for the number of Markoff surfaces such that there is no Brauer--Manin obstruction to the integral Hasse principle, as done in \cite{CTWX20} (slightly better than the result $M^{1/2}/\log M$ in \cite{LM20}). 

We begin our study in the case of Markoff-type K3 surfaces by the following two results.

\begin{proposition}
For $k \in \Zbb$, let $W_{k} \subset \Pbb^{1} \times \Pbb^{1} \times \Pbb^{1}$ be the MK3 surfaces defined over $\Qbb$ by the $(2,2,2)$-form 
\begin{equation}
    F_{3}(x,y,z) = x^{2} + y^{2} + z^{2} + 4(x^{2}y^{2} + y^{2}z^{2} + z^{2}x^{2}) - 16x^{2}y^{2}z^{2} - k = 0.
\end{equation}
Denote by $\mathcal{U}_{k}$ the integral model of $U_{k}$ defined over $\Zbb$ by the same equation. If $k$ satisfies the conditions:
\begin{enumerate}
    \item $k = \ell(\ell+1)$ where $\ell \in \Zbb$ such that $\ell \equiv 5$ \textup{mod} $8$, $\ell \equiv 4$ \textup{mod} $27$, $\ell \equiv 1$ \textup{mod} $35$, and $\ell \not \equiv 0, -1$ \textup{mod} $37$;
    \item $p \equiv \pm 1, 3$ \textup{mod} $8$ for any prime divisor $p$ of $2\ell + 1$,
\end{enumerate}
then there is a Brauer--Manin obstruction to the integral Hasse principle on $\mathcal{U}_{k}$ with respect to the subgroup $\mathcal{A} \subset \textup{Br}_{1}\,U_{k}/\textup{Br}_{0}\,U_{k}$ generated by the elements $\mathcal{A}_{1} = (4x^{2}+1, 2(4k+1))$ and $\mathcal{A}_{2} = (4y^{2}+1, 2(4k+1))$, i.e., $\mathcal{U}_{k}(\Zbb) \subset \mathcal{U}_{k}(\textbf{\textup{A}}_{\Zbb})^{\mathcal{A}} = \emptyset$.
\end{proposition}

\begin{proof}
The proof is similar as usual, with notice that only the local invariant at $p = 2$ is nonzero which makes the total sum of invariants nonzero, hence a contradiction.
\end{proof}

\begin{proposition}
For $k \in \Zbb$, let $W_{k} \subset \Pbb^{1} \times \Pbb^{1} \times \Pbb^{1}$ be the MK3 surfaces defined over $\Qbb$ by the $(2,2,2)$-form 
\begin{equation}
    F_{3}(x,y,z) = x^{2} + y^{2} + z^{2} + 4(x^{2}y^{2} + y^{2}z^{2} + z^{2}x^{2}) - 16x^{2}y^{2}z^{2} - k = 0.
\end{equation}
Denote by $\mathcal{U}_{k}$ the integral model of $U_{k}$ defined over $\Zbb$ by the same equation. If $k$ satisfies the conditions:
\begin{enumerate}
    \item $k = \ell(\ell+1)$ where $\ell \in \Zbb$ such that $\ell \equiv 3$ \textup{mod} $8$, $\ell \equiv 4$ \textup{mod} $27$, $\ell \equiv 1$ \textup{mod} $35$, and $\ell \not \equiv 0, -1$ \textup{mod} $37$;
    \item $p \equiv \pm 1, 3$ \textup{mod} $8$ for any prime divisor $p$ of $2\ell + 1$,
\end{enumerate}
then there is \textbf{no} Brauer--Manin obstruction to the integral Hasse principle on $\mathcal{U}_{k}$ with respect to the subgroup $\mathcal{A} \subset \textup{Br}_{1}\,U_{k}/\textup{Br}_{0}\,U_{k}$ generated by the elements $\mathcal{A}_{1} = (4x^{2}+1, 2(4k+1))$ and $\mathcal{A}_{2} = (4y^{2}+1, 2(4k+1))$, i.e., $\mathcal{U}_{k}(\textbf{\textup{A}}_{\Zbb})^{\mathcal{A}} \not= \emptyset$.
\end{proposition}

\begin{proof}
The proof is similar as above, except that with $k = \ell(\ell+1) \equiv 4$ mod $8$, the local invariants at $p = 2$ are $(0,0)$, which makes the total sum of invariants always zero, hence the conclusion. In fact, it even shows that $\mathcal{U}_{k}(\textbf{\textup{A}}_{\Zbb})^{\mathcal{A}} = \mathcal{U}_{k}(\textbf{\textup{A}}_{\Zbb})$. 
\end{proof}

\begin{remark}
The first Proposition is only used to give a different family of Markoff-type K3 surfaces for which there is a Brauer--Manin obstruction and to make an interesting comparison with the second Proposition. In fact, one may give an elementary proof for the fact that the set of integral points is empty as follows. 

Assume that there is an integral point $(x,y,z) \in \mathcal{U}_{k}(\Zbb)$, then if $|x|,|y|,|z| \geq 1$, $F(x,y,z) < 0$ as $k = \ell(\ell+1) > 0$. Therefore, at least one of $x,y,z$ must be zero, and without loss of generality, we may assume that $z=0$. The equation is equivalent to 
$$ (4x^{2}+1)(4y^{2}+1) = (2\ell+1)^{2}. $$
As the right hand side is divisible by $3$ since $\ell \equiv 4$ mod $27$, so is the left hand side. However, this is a contradiction as $-1$ is not a square modulo $3$.
\\~\\
\indent The second Proposition only gives the result with respect to a proper subgroup of the Brauer group since we are not able to determine the whole (algebraic) Brauer--Manin set to prove whether it is nonempty or not. That is also the reason why we have not yet found a similar counting result to the ones in \cite{LM20} and \cite{CTWX20}.
\end{remark}

\subsection{Failure of strong approximation}
Next, we consider some cases when strong approximation, instead of the integral Hasse principle, fails, while integral points can exist.

\begin{proposition}
For $k \equiv 2$ mod $8$, let $W_{k} \subset \Pbb^{1} \times \Pbb^{1} \times \Pbb^{1}$ be the MK3 surfaces defined over $\Qbb$ by the $(2,2,2)$-form 
\begin{equation}
    F_{2}(x,y,z) = x^{2} + y^{2} + z^{2} - 4(x^{2}y^{2} + y^{2}z^{2} + z^{2}x^{2}) + 16x^{2}y^{2}z^{2} - k = 0.
\end{equation}
Let $\mathcal{U}_{k}$ be the integral model of $U_{k}$ defined over $\Zbb$ by the same equation. If $\mathcal{U}_{k}(\Zbb) \not= \emptyset$, while there is a Brauer--Manin obstruction to strong approximation on $\mathcal{U}_{k}$ with respect the element $\mathcal{A}_{1} = (4x^{2}-1, k)$, i.e., $\mathcal{U}_{k}(\textbf{\textup{A}}_{\Zbb})^{\mathcal{A}_{1}} \not= \mathcal{U}_{k}(\textbf{\textup{A}}_{\Zbb})$.
\end{proposition}
To illustrate our choice of $k$, we can choose an integral point $(x,y,z) = (1,1,8) \in \mathcal{U}_{k}(\Zbb)$ to have $k = 574$.

\begin{proof}
Assume that we have $(x,y,z) \in \mathcal{U}(\Zbb)$, so with $k \equiv 2$ mod $8$ we can assume further without loss of generality that $x, y$ are odd and $z = 2a$ is even. Since $\mathcal{U}_{k}(\Zbb) \subset \mathcal{U}_{k}(\textbf{\textup{A}}_{\Zbb})^{\mathcal{A}_{1}}$, the set $\mathcal{U}_{k}(\textbf{\textup{A}}_{\Zbb})^{\mathcal{A}_{1}}$ is nonempty, and so is $\mathcal{U}_{k}(\textbf{\textup{A}}_{\Zbb})$. Viewing $(x,y,z)$ as an element of $\mathcal{U}_{k}(\textbf{\textup{A}}_{\Zbb})$ via the diagonal embedding, we can find another local integral point $(x',y',z')$ with the $2$-part $(x'_{2},y'_{2},z'_{2}) = (z_{2},x_{2},y_{2})$ and the same $p$-parts as those of $(x,y,z)$ for every $p \not= 2$, so that $\textup{inv}_{2}\,\mathcal{A}_{1}(x',y',z') = (4.4a^{2}-1,k)_{2} = 0 \not= 1/2 = (3,k))_{2} = (4x_{2}^{2}-1,k)_{2}$. Consequently,
$$ \sum_{p} \textup{inv}_{p}\,\mathcal{A}_{1}(x',y',z') \not= \sum_{p} \textup{inv}_{p}\,\mathcal{A}_{1}(x,y,z) = 0. $$
Therefore, $(x',y',z') \not\in \mathcal{U}_{k}(\textbf{\textup{A}}_{\Zbb})^{\mathcal{A}_{1}}$ and the result follows.
\end{proof}

\begin{proposition}
For $k \equiv 1$ mod $4$, let $W_{k} \subset \Pbb^{1} \times \Pbb^{1} \times \Pbb^{1}$ be the MK3 surfaces defined over $\Qbb$ by the $(2,2,2)$-form 
\begin{equation}
    F_{3}(x,y,z) = x^{2} + y^{2} + z^{2} + 4(x^{2}y^{2} + y^{2}z^{2} + z^{2}x^{2}) - 16x^{2}y^{2}z^{2} - k = 0.
\end{equation}
Let $\mathcal{U}_{k}$ be the integral model of $U_{k}$ defined over $\Zbb$ by the same equation. If $\mathcal{U}_{k}(\Zbb) \not= \emptyset$, then there is a Brauer--Manin obstruction to strong approximation on $\mathcal{U}_{k}$ with respect to the element $\mathcal{A}_{1} = (4x^{2}+1, -2(4k+1))$, i.e., $\mathcal{U}_{k}(\textbf{\textup{A}}_{\Zbb})^{\mathcal{A}_{1}} \not= \mathcal{U}_{k}(\textbf{\textup{A}}_{\Zbb})$.
\end{proposition}
To illustrate our choice of $k$, we can choose an integral point $(x,y,z) = (1,4,4) \in \mathcal{U}_{k}(\Zbb)$ to have $k = -2911$. 

\begin{proof}
Assume that we have $(x,y,z) \in \mathcal{U}(\Zbb)$, so with $k \equiv 1$ mod $4$ we can assume further without loss of generality that $x$ is odd and $y = 2a, z = 2b$ are even. Since $\mathcal{U}_{k}(\Zbb) \subset \mathcal{U}_{k}(\textbf{\textup{A}}_{\Zbb})^{\mathcal{A}_{1}}$, the set $\mathcal{U}_{k}(\textbf{\textup{A}}_{\Zbb})^{\mathcal{A}_{1}}$ is nonempty, and so is $\mathcal{U}_{k}(\textbf{\textup{A}}_{\Zbb})$. Viewing $(x,y,z)$ as an element of $\mathcal{U}_{k}(\textbf{\textup{A}}_{\Zbb})$ via the diagonal embedding, we can find another local integral point $(x',y',z')$ with the $2$-part $(x'_{2},y'_{2},z'_{2}) = (y_{2},x_{2},z_{2})$ and the same $p$-parts as those of $(x,y,z)$ for every $p \not= 2$, so that $\textup{inv}_{2}\,\mathcal{A}_{1}(x',y',z') = (4.4a^{2}+1,-2(4k+1))_{2} = 0 \not= 1/2 = (5,-2(4k+1))_{2} = (4x_{2}^{2}+1,-2(4k+1))_{2}$. Consequently,
$$ \sum_{p} \textup{inv}_{p}\,\mathcal{A}_{1}(x',y',z') \not= \sum_{p} \textup{inv}_{p}\,\mathcal{A}_{1}(x,y,z) = 0. $$
Therefore, $(x',y',z') \not\in \mathcal{U}_{k}(\textbf{\textup{A}}_{\Zbb})^{\mathcal{A}_{1}}$ and the result follows.
\end{proof}

\begin{remark}
In the case of $F_{1}$, let $W_{k} \subset \Pbb^{1} \times \Pbb^{1} \times \Pbb^{1}$ be the MK3 surfaces defined over $\Qbb$ by the $(2,2,2)$-form 
\begin{equation}
    F_{1}(x,y,z) = x^{2} + y^{2} + z^{2} - 4x^{2}y^{2}z^{2} - k = 0
\end{equation}
and $\mathcal{U}_{k}$ be the integral model of $U_{k}$ defined over $\Zbb$ by the same equation. When $\mathcal{U}_{k}(\Zbb) \not= \emptyset$, it seems likely that the local invariant at $p = 2$ of the Brauer element $\mathcal{A} = (4x^{2}y^{2}-1, k+1) = (4y^{2}z^{2}-1, k+1) = (4z^{2}x^{2}-1, k+1)$ is constant for various choices of $2$-adic integral points. In other words, we may need to work with other primes to (possibly) find a Brauer--Manin obstruction to strong approximation.
\end{remark}

\subsection{Rational points on affine surfaces}
Finally, we study the existence of rational points on affine Markoff-type K3 surfaces. For Markoff surfaces, we know from \cite{Kol02}, \cite{LM20} and \cite{CTWX20} that there are always rational points on smooth affine Markoff sufaces; this comes from the fact that any smooth cubic surface over an infinite field $k$ is $k$-unirational as soon as it has a $k$-rational point. However, such a phenomenon does not happen for smooth affine MK3 surfaces, since their projective closures are elliptic surfaces and lie in $(\Pbb^{1})^{3}$ instead of $\Pbb^{3}$. We know there are always rational points at infinity for our families of MK3 surfaces, but we are not certain whether there are also rational points on the affine open subscheme or not. As a modest contribution to the existence problem of rational points, we have the following result.

\begin{proposition}
For $k \in \Zbb$, let $W_{k} \subset \Pbb^{1} \times \Pbb^{1} \times \Pbb^{1}$ be the MK3 surfaces defined over $\Qbb$ by the $(2,2,2)$-form 
\begin{equation}
    F_{1}(x,y,z) = x^{2} + y^{2} + z^{2} - 4x^{2}y^{2}z^{2} - k = 0.
\end{equation}
Let $\mathcal{U}_{k}$ be the integral model of $U_{k}$ defined over $\Zbb$ by the same equation. If $k$ satisfies the conditions:
\begin{enumerate}
    \item $k = -(1 + 16\ell^{2})$ where $\ell \in \Zbb$ such that $\ell$ is odd and $\ell \not\equiv \pm 2$ \textup{mod} $5$;
    \item $p \equiv 1$ \textup{mod} $4$ for any prime divisor $p$ of $\ell$,
\end{enumerate}
then there is no Brauer--Manin obstruction to the (rational) Hasse principle on $U_{k}$ with respect to the element $\mathcal{A} = (4x^{2}y^{2}-1, k+1) = (4y^{2}z^{2}-1, k+1) = (4z^{2}x^{2}-1, k+1)$ in $\textup{Br}_{1}\,U_{k}/\textup{Br}_{0}\,U_{k}$. In other words, $U_{k}(\textbf{\textup{A}}_{\Qbb})^{\mathcal{A}} \not= \emptyset$.
\end{proposition}

\begin{proof}
The proof proceeds similarly as usual, with notice that for $p = 2$, besides the local integral point lifted from $(1,1,1) \in \mathcal{U}_{k}(\Zbb/8\Zbb)$ which gives the local invariants $(1/2,1/2)$, there exists another local point $(x_{2},y_{2},z_{2}) \in U_{k}(\Qbb_{2})$ with $\textup{v}_{2}(x_{2}) = -1, \textup{v}_{2}(y_{2}) = -3$, and $\textup{v}_{2}(z_{2}) = 0$ which gives the local invariants $(0,0)$.
\end{proof}

\begin{proposition}
For $k \in \Zbb$, let $W_{k} \subset \Pbb^{1} \times \Pbb^{1} \times \Pbb^{1}$ be the MK3 surfaces defined over $\Qbb$ by the $(2,2,2)$-form 
\begin{equation}
    F_{2}(x,y,z) = x^{2} + y^{2} + z^{2} - 4(x^{2}y^{2} + y^{2}z^{2} + z^{2}x^{2}) + 16x^{2}y^{2}z^{2} - k = 0.
\end{equation}
Let $\mathcal{U}_{k}$ be the integral model of $U_{k}$ defined over $\Zbb$ by the same equation. If $k$ satisfies the conditions:
\begin{enumerate}
    \item $k = 18\ell^{2}$ where $\ell \in \Zbb$ such that $\ell \not\equiv 0$ \textup{mod} $2,3$, $\ell \equiv 1$ \textup{mod} $5$, and $\ell \equiv 2$ \textup{mod} $7$;
    \item $p \equiv \pm 1$ \textup{mod} $8$ for any prime divisor $p$ of $\ell$,
\end{enumerate}
then there is no Brauer--Manin obstruction to the (rational) Hasse principle on $U_{k}$ with respect to the subgroup $A \subset \textup{Br}_{1}\,U_{k}/\textup{Br}_{0}\,U_{k}$ generated by the elements $\mathcal{A}_{1} = (4x^{2}-1, k)$ and $\mathcal{A}_{2} = (4y^{2}-1, k)$, i.e., $U_{k}(\textbf{\textup{A}}_{\Qbb})^{A} \not= \emptyset$.
\end{proposition}

\begin{proof}
The proof proceeds similarly as usual, with notice that for $p = 3$, besides the local integral point lifted from $(3,3,0) \in \mathcal{U}_{k}(\Zbb/27\Zbb)$ which gives the local invariants $(0,0)$, there exists another local point $(x_{3},y_{3},z_{3}) \in U_{k}(\Qbb_{3})$ with $\textup{v}_{3}(x_{3}) = \textup{v}_{3}(y_{3}) = 0, \textup{v}_{3}(z_{3}) = -1$ and $2z_{3} = \frac{a}{b}$ such that $a \equiv 2, b \equiv 3$, $x_{3} \equiv -2, y_{3} \equiv 13$ (mod $27$): this gives the local invariants $(1/2,1/2)$.
\end{proof}

\begin{proposition}
For $k \in \Zbb$, let $W_{k} \subset \Pbb^{1} \times \Pbb^{1} \times \Pbb^{1}$ be the MK3 surfaces defined over $\Qbb$ by the $(2,2,2)$-form 
\begin{equation}
    F_{3}(x,y,z) = x^{2} + y^{2} + z^{2} + 4(x^{2}y^{2} + y^{2}z^{2} + z^{2}x^{2}) - 16x^{2}y^{2}z^{2} - k = 0.
\end{equation}
Let $U_{k} \subset W_{k}$ be the affine open subscheme defined by $\{rst \not= 0\}$ over $\Qbb$ by the same equation. If $k$ satisfies the conditions:
\begin{enumerate}
    \item $k = -\frac{1}{4}(1 + 27\ell^{2})$ where $\ell \in \Zbb$ such that $\ell \equiv \pm 1$ \textup{mod} $8$, $\ell \equiv 1$ \textup{mod} $5$, $\ell \equiv 3$ \textup{mod} $7$, and $\ell \not\equiv \pm 10$ \textup{mod} $37$;
    \item $p \equiv \pm 1$ \textup{mod} $24$ for any prime divisor $p$ of $\ell$,
\end{enumerate}
then there is no Brauer--Manin obstruction to the (rational) Hasse principle on $U_{k}$ with respect to the subgroup $A \subset \textup{Br}_{1}\,U_{k}/\textup{Br}_{0}\,U_{k}$ generated by the elements $\mathcal{A}_{1} = (4x^{2}+1, -2(4k+1))$ and $\mathcal{A}_{2} = (4y^{2}+1, -2(4k+1))$, i.e., $U_{k}(\textbf{\textup{A}}_{\Qbb})^{A} \not= \emptyset$.
\end{proposition}

\begin{proof}
The proof proceeds similarly as usual, with notice that for $p = 3$, besides the local integral point lifted from $(1,1,1) \in \mathcal{U}_{k}(\Zbb/3\Zbb)$ which gives the local invariants $(1/2,1/2)$, there exists another local point $(x_{3},y_{3},z_{3}) \in U_{k}(\Qbb_{3})$ with $\textup{v}_{3}(x_{3}) < 0, \textup{v}_{3}(y_{3}) < 0$, and $\textup{v}_{3}(z_{3}) = 0$ which gives the local invariants $(0,0)$.
\end{proof}

\begin{remark}
Once again, we do not know whether the Brauer--Manin set with respect to the whole Brauer group is nonempty or not, but at least we know that there is no Brauer--Manin obstruction to the existence of rational points with respect to the Brauer subgroups that we are interested in. We believe that there should exist rational points on those families of affine MK3 surfaces, but we do not know how to prove or disprove this claim in general.
\end{remark}

\begin{example}
We consider the first surface in Example 4.1 where the affine MK3 surface in question contains no integral points (due to Brauer--Manin obstruction as previously shown) but indeed contains rational points. Let $W_{-17} \subset \Pbb^{1} \times \Pbb^{1} \times \Pbb^{1}$ be the MK3 surfaces defined over $\Qbb$ by the $(2,2,2)$-form 
\begin{equation}
    F_{1}(x,y,z) = x^{2} + y^{2} + z^{2} - 4x^{2}y^{2}z^{2} + 17 = 0.
\end{equation}
Denote by $\mathcal{U}_{-17}$ the integral model of $U_{-17}$ defined over $\Zbb$ by the same equation. Then $\mathcal{U}_{-17}(\Zbb) = \emptyset$; however, we can find a few rational points of small height in $U_{-17}(\Qbb)$ using SageMath \cite{SJ05}: (1/2, 49/24, 13/5), (1/3, 5/2, 29/8), (22/25, 23/16, 23/12), (27/29, 47/34, 15/8), (7/32, 46/15, 23/4).
\end{example}

\textsc{Sorbonne Université and Université Paris Cité, CNRS, IMJ-PRG, F-75005 Paris, France}\\
\textit{E-mail address}: \href{mailto:quang-duc.dao@imj-prg.fr}{\texttt{quang-duc.dao@imj-prg.fr}}

\end{document}